\documentclass[12pt,a4paper]{amsart}

\usepackage{amssymb}
\usepackage{amsthm}
\usepackage{amsmath}
\usepackage{a4wide}
\usepackage{graphicx}
\let\oldemptyset\emptyset
\usepackage{mathabx}
\let\emptyset\oldemptyset
\usepackage{enumerate}
\usepackage{xcolor}
\usepackage{tikz}
\usepackage{graphbox}
\usepackage{orcidlink}
\usepackage{hyperref}
\hypersetup{colorlinks=true,allcolors=blue}

\newtheorem{theorem}{Theorem}[section]
\newtheorem{proposition}[theorem]{Proposition}
\newtheorem{corollary}[theorem]{Corollary}
\newtheorem{lemma}[theorem]{Lemma}

\theoremstyle{definition}
\newtheorem{example}[theorem]{Example}
\newtheorem{definition}[theorem]{Definition}
\newtheorem{remark}[theorem]{Remark}

\newcommand{\II}{{\mathbb I}}
\newcommand{\RR}{\mathbb{R}}
\newcommand{\kd}{K_\delta}
\newcommand{\ad}{A_\delta}
\newcommand{\bd}{B_\delta}
\newcommand{\ud}{U_\delta}
\newcommand{\cc}{\overline{C}_\delta}
\newcommand{\dhat}{\widehat{\delta}}
\newcommand{\tv}{\mathrm{TV}}
\newcommand{\NN}{\mathbb{N}}

\newcommand{\vol}{\mathrm{V}}

\title{Exact upper bound for copulas with a given diagonal section}

\author[D. Kokol Bukovšek]{Damjana Kokol Bukovšek \orcidlink{0000-0002-0098-6784} }
\address{University of Ljubljana, School of Economics and Business, and Institute of Mathematics, Physics and Mechanics, Ljubljana, Slovenia}
\email{Damjana.Kokol.Bukovsek@ef.uni-lj.si}

\author[B. Mojškerc]{Blaž Mojškerc \orcidlink{0000-0001-8096-355X} }
\address{University of Ljubljana, School of Economics and Business, and Institute of Mathematics, Physics and Mechanics, Ljubljana, Slovenia}
\email{Blaz.Mojskerc@ef.uni-lj.si}

\author[N. Stopar]{Nik Stopar \orcidlink{0000-0002-0004-4957} } 
\address{University of Ljubljana, Faculty of Civil and Geodetic Engineering, and Institute of Mathematics, Physics and Mechanics, Ljubljana, Slovenia}
\email{Nik.Stopar@fgg.uni-lj.si}


\begin{document}

\maketitle

\begin{abstract}
   We answer a 15-year-old open question about the exact upper bound for bivariate copulas with a given diagonal section by giving an explicit formula for this bound. As an application, we determine the maximal asymmetry of bivariate copulas with a given diagonal section and construct a copula that attains it. We derive a formula for the maximal asymmetry that is simple enough to be used by practitioners.
\end{abstract}

\section{Introduction}
In recent years there has been a lot of interest in determining exact upper and lower bounds for sets of copulas satisfying some additional conditions.
The main motivation for these studies are practical applications in situations where the exact joint distribution of studied random variables is unknown, and only partial information is available. Through Sklar's theorem this partial information can usually be translated into conditions for the corresponding copulas. In this case, having tight bounds for the set of all admissible copulas allows one to make better estimations of the quantities of interest.

A particular example of the above is the situation when we prescribe the values of copulas on a given subset $S$ of the unit box.
Standard cases for the set $S$ that have been considered in the literature include a finite set of points, a vertical or horizontal section, the main or opposite diagonal, a curvilinear section, a track, or a general compact set.

Exact upper and lower bounds for bivariate copulas with prescribed value at a single point are given in \cite[Theorem~3.2.3]{Ne99}. In the multivariate setting these bounds were studied in  \cite{RoUb05} and finally determined in \cite{KlEA22}. For copulas with prescribed values at more than one point see \cite{MaSaSh10,DeDe13,KlEA22}.

Bivariate copulas with prescribed diagonal section $\delta$ were first studied by Bertino \cite{Be77}, and later by Fredricks and Nelsen \cite{FrNe97}.
The exact lower bound for such copulas is given in  \cite{FrNe02} and it is the Bertino copula.
In \cite[Theorems~20 and 32]{NeQuRoUb} an upper bound is provided which is exact in some special cases. An example is also given which shows that this bound is not exact in general.
The exact upper bound for all bivariate copulas with a given diagonal section is still an open question which we settle in this paper.

Diagonal sections of multivariate copulas were investigated in \cite{CuTh01,Ja09}. 
Bivariate copulas with prescribed opposite diagonal sections were considered in \cite{deAmEA16}.
In \cite{ZouEA} the authors study bounds for copulas with prescribed curvilinear sections.
For more information on these results we refer the reader to a review paper \cite{FeUb}.
Copulas with prescribed vertical or horizontal section are investigated in \cite{DuKoMeSe07}, while tracks of copulas are considered in \cite{GeQuRoSe99}. Some results on copulas with prescribed values on general compact sets can be found in \cite{LuPa17}.

The main goal of this paper is to answer an open question of the exact upper bound for bivariate copulas with a given diagonal section (see \cite{NeQuRoUb}) by giving an explicit formula for this bound.
We achieve this by constructing a new copula with prescribed diagonal section, which attains the bound on the entire upper-left triangle of the unit square. We also answer the question posed in \cite{NeQuRoUb}, for which diagonal sections this exact bound is a copula.

As an application of our main result, we determine the maximal asymmetry of bivariate copulas with a given diagonal section and construct a copula that attains it.
In order to achieve this, we rely heavily on the formulas for the exact upper and lower bounds for copulas with a given diagonal section.
We simplify originally complicated optimization problem of finding the maximal asymmetry on non-tangible, partly
unknown domain to simpler optimization problems on compact intervals or the upper-left triangle of the unit square. In particular, we give explicit formulas to calculate the maximal asymmetry value, which should be simple enough to be used in practice. We also give some examples to illustrate our formulas.

The asymmetry, also known as the degree of non-exchangeability, is important in practice, as it is often the case that the data sets exhibit asymmetric properties \cite{GeNe13}.
To choose an appropriate family of copulas to model the data, it is advantageous to examine the asymmetry of copulas.
The asymmetry of copulas has been extensively studied, see for example \cite{KlMe,Ne07,DuPa10,KoKoMoOm20}.

This paper is structured as follows. In Section~\ref{sec:pre} we give the necessary definitions and recall some known results on copulas with prescribed diagonal section. In Section~\ref{sec:cop} we determine the exact upper bound of all such copulas, and in Section~\ref{sec:asymmetry} we apply our results to calculate their maximal asymmetry.

\section{Preliminaries}\label{sec:pre}

Throughout the paper we shall denote the unit interval by $\II=[0,1]$. We will use the terms \emph{increasing} to mean non-decreasing and \emph{decreasing} to mean non-increasing.

\begin{definition}
Let   $C\colon\II^2\to\II$ be a bivariate function. Then
\begin{itemize}
    \item[(i)] $C$ is \emph{grounded} if $C(x, 0)=C(0,y)=0$ for all $x,y \in \II$;
    \item[(ii)] $C$ has \emph{uniform marginals} if $C(x,1)=x$ and $C(1,y)=y$ for all $x,y \in \II$;
    \item[(iii)] $C$ is \emph{$2$-increasing} if for any rectangle $R=[a,b]\times [c,d] \subseteq \II^2$ it holds that
\begin{equation*}
    \vol_C(R) = C(b,d) + C(a,c) - C(b,c) - C(a,d) \ge 0.
\end{equation*}
\end{itemize}
\end{definition}

\begin{definition}
A bivariate function $C\colon\II^2\to\II$ is called
\begin{itemize}
\item[(i)] a \emph{copula} if it is grounded, has uniform marginals and is $2$-increasing;
\item[(ii)] a \emph{quasi-copula} if it is grounded, has uniform marginals, is increasing in each variable, and $1$-Lipschitz.
\end{itemize}
\end{definition}

We denote by $\mathcal{C}$ the set of all copulas and by $\mathcal{Q}$ the set of all quasi-copulas. Set $\mathcal{C}$ is a subset of $\mathcal{Q}$ and set $\mathcal{Q}$ is closed under taking arbitrary point-wise infima and suprema, i.e., it is a complete lattice, see \cite{NeLaUbFl04,NeUb05}.
If $C$ is a copula or a quasi-copula we denote by $\delta_C$ its \emph{diagonal section}, i.e., a function $\delta_C(x) = C(x,x)$, and by $C^t$ its \emph{transpose}, i.e., copula or quasi-copula $C^t(x,y) = C(y,x)$.

We have the following characterization of diagonal sections (for copulas see \cite{FrNe97}, the proof for quasi-copulas is the same).

\begin{proposition}\label{pr:delta}
For a function $\delta\colon \II \to \II$ the following conditions are equivalent:
\begin{enumerate}[$(i)$]
    \item $\delta$ is a diagonal section of some copula,
    \item $\delta$ is a diagonal section of some quasi-copula,
    \item $\delta$ satisfies the conditions
    \begin{enumerate}[$(a)$]
        \item $\delta(x)\le x$ for all $x\in \II$,
        \item $0\le \delta(y)- \delta(x) \le 2(y-x)$ for all $x,y \in \II$ such that $x \le y$, and
        \item $\delta(1)=1$.
    \end{enumerate}
\end{enumerate}
\end{proposition}

For any diagonal section $\delta$ we denote by $\dhat\colon \II \to \II$ the function defined by $\dhat(x)=x- \delta(x)$. 
The function $\dhat$ is 1-Lipschitz, since, by Proposition~\ref{pr:delta} $(iii) (b)$, we have for any $x,y \in \II$ with $y \ge x$
$$ \dhat(y) - \dhat(x) = (y-x) - (\delta(y) - \delta(x)) \le y-x $$
and 
$$ \dhat(y) - \dhat(x) = (y-x) - (\delta(y) - \delta(x)) \ge y-x - 2(y-x) = x-y.$$
This fact will be often used in our proofs. A similar argument as above shows that the difference of two increasing 1-Lipschitz functions is also 1-Lipschitz.

\begin{definition}
For any copula or quasi-copula $C$ its \emph{asymmetry} is defined as the supremum norm of the difference $C - C^t$, i.e.,
$$ \mu(C) = \max\big\{|C(x,y)-C(y,x)| \colon x,y\in \II\big\}.$$
\end{definition}

The asymmetry of any copula $C$ lies in the interval $[0,\frac13]$, see \cite{KlMe,Ne07}. If $\mu(C) =0$, then $C=C^t$ and $C$ is called \emph{symmetric}. 

\begin{definition} For a diagonal section $\delta$ define functions $\ad, \bd, \cc, \underline{C}_\delta, \kd \colon \II^2 \to \II$ by
    \begin{enumerate}[(i)]
        \item $\displaystyle \ad(x,y)= \sup \{ Q(x,y) \colon Q\in \mathcal{Q}, \delta_Q=\delta \}$,
        \item $\displaystyle \bd(x,y) = \inf \{ Q(x,y) \colon Q \in \mathcal{Q}, \delta_Q= \delta \}$,
        \item $\displaystyle \cc(x,y)= \sup \{ C(x,y) \colon C \in \mathcal{C}, \delta_C = \delta \},$
         \item $\displaystyle \underline{C}_\delta(x,y) = \inf \{ C(x,y) \colon C \in \mathcal{C}, \delta_C= \delta \}$,
        \item $\displaystyle \kd(x,y)=\sup \{ C(x,y) \colon C\in \mathcal{C}, C^t=C, \delta_C = \delta \}$.
    \end{enumerate}
\end{definition}

There exist explicit formulas for four of these functions, which we now recall. The formula for the remaining one will be given in our Theorem~\ref{th:cc}.
For any $x,y \in \II$ we will sometimes denote $x \land y=\min\{x,y\}$ and  $x \lor y=\max\{x,y\}$.

\begin{theorem}[\cite{NeLaUbFl04,FrNe97}] \label{th:ABK}
For a diagonal section $\delta$ we have
    \begin{enumerate}[$(i)$]
        \item $\displaystyle \ad(x,y)= \min \big\{ x, y, \max \{ x,y \} - \max_{t \in [x \land y, x \lor y]} \dhat(t)  \big\}$,
        \item $\displaystyle \bd(x,y) = \underline{C}_\delta(x,y) = \min \{x,y \} -\min_{t \in [x \land y, x \lor y]} \dhat(t)$,
        \item $\displaystyle \kd(x,y)= \min \big\{x,y, \frac 12 \big(\delta(x)+\delta(y)\big) \big\}$.
    \end{enumerate}
\end{theorem}

Functions $\ad, \bd, \cc, \underline{C}_\delta, \kd$ are all quasi-copulas, since they are infima or suprema
of quasi-copulas. It turns out that functions $\bd$ and $\kd$ are always copulas, see \cite{NeLaUbFl04,FrNe97}.

The copula $\bd$ is called \emph{Bertino} copula and it is always singular. The authors in \cite{FrNe02} describe its support. They define a function $h^\delta\colon \II \to \II$ by
\begin{equation}\label{eq:h_delta}
h^\delta(x) = \max\{y \in \II \colon y \ge x, \dhat(t) \ge \dhat(x) \text{ for all } t \in [x,y]\}.
\end{equation}
Note that the maximum is attained since $\dhat$ is continuous.
The function $h^\delta$ is strictly decreasing on every interval where $\dhat$ is strictly increasing, $h^\delta$ is constant on every interval where $\dhat$ is constant, $h^\delta$ is an identity on every open interval where $\dhat$ is strictly decreasing, see \cite{FrNe02}.
In addition, since $\dhat$ is continuous, $\dhat(h^\delta(x))=\dhat(x)$ and $h^\delta(h^\delta(x))=h^\delta(x)$ for all $x \in \II$. 
Furthermore, $h^\delta$ is upper-semicontinuous, i.e., for every $x \in \II$ we have $\limsup_{u\to x}h^\delta(u) \le h^\delta(x)$. Indeed, we have
$$\limsup_{u \to x} h^\delta(u) \ge \limsup_{u \to x} u=x.$$
If $\limsup_{u \to x} h^\delta(u) = x$, then $\limsup_{u \to x} h^\delta(u) \le h^\delta(x)$ by definition and the claim holds. Otherwise, let $t \in (x, \limsup_{u \to x} h^\delta(u))$. There exists a sequence $u_i \in \II$ such that $\lim_{i \to \infty} u_i=x$ and $\lim_{i \to \infty} h^\delta(u_i) = \limsup_{u \to x} h^\delta(u)$.
Then $u_i<t<h^\delta(u_i)$ for $i$ sufficiently large, so that $\dhat(t) \ge \dhat(u_i)$ by \eqref{eq:h_delta}.
Taking the limit as $i$ goes to $\infty$ gives $\dhat(t) \ge \dhat(x)$.
Since also
$$\dhat(\limsup_{u \to x} h^\delta(u))=\limsup_{u \to x} \dhat(h^\delta(u))=\limsup_{u \to x} \dhat(u)=\dhat(x),$$
we have $\dhat(t) \ge \dhat(x)$ for all $t \in [x,\limsup_{u \to x} h^\delta(u)]$, which implies $\limsup_{u\to x}h^\delta(u) \le h^\delta(x)$ and the claim is proved.

\begin{theorem}[\cite{FrNe02}] \label{th:Bertino}
For a diagonal section $\delta$ the support of the Bertino copula $\bd$ is the smallest closed set which is symmetric with respect to the main diagonal and contains the continuous, strictly decreasing parts of the graph of function $h^\delta$, and the closure of the set $\{(x,x) \in \II^2\colon \dhat'(x) \ne \pm1\}$.
\end{theorem}

We now recall the definition of total variation. 

\begin{definition} For a function $f \colon \II \to \mathbb{R}$ and $0 \le x \le y \le 1$ the \emph{total variation} of $f$ on $[x,y]$ is defined by
    $$\displaystyle \tv_x^y(f) = \sup \left\{ \sum_{i=1}^{n} |f(x_{i}) - f(x_{i-1})| \colon x=x_0 < x_1 < \ldots < x_n = y,\ n \in \NN \cup \{0\} \right\}.$$
\end{definition}

For convenience we extend the definition of $\tv_x^y(f)$ to the case when $y<x$ by letting
$\tv_x^y(f)=-\tv_y^x(f)$. If $\tv_0^1(f)$ is finite then 
$\tv_x^y(f)=\tv_0^y(f)-\tv_0^x(f)$ for any $x,y \in \II$.
For a $1$-Lipschitz function $f$ we have $|\tv_x^y(f)| \le |y-x|$, in particular, the function $x \mapsto \tv_0^x(f)$ is $1$-Lipschitz and increasing as a function $\II \to \mathbb{R}$.
For any absolutely continuous function $f$, in particular, for any $1$-Lipschitz function, the following formula holds \cite[Theorem 7.31]{WhZy}:
\begin{equation}\label{eq:tvint}
    \tv_x^y(f)= \int_x^y |f'(t)|dt.
\end{equation}

\section{Copulas with prescribed diagonal section}\label{sec:cop}

We first prove a new upper bound for all copulas with prescribed diagonal section. Later we will prove that it is exact.

\begin{proposition}\label{pr:bound}
Let $\delta$ be a diagonal section and let $C\in \mathcal{C}$ be any copula with $\delta_C=\delta$. Then for every $x,y \in \II$ it holds that
\begin{equation}\label{eq:zgornja}
    C(x,y) \le  \min \left\{ x,y,\max \{ x,y \} - \frac12 \left(\dhat(x)+ \dhat(y) + \tv_{x \land y}^{x\lor y} (\dhat) \right) \right\}.
\end{equation}
\end{proposition}

\begin{remark}
    The intuition behind the bound in \eqref{eq:zgornja} is as follows. In \cite[Theorem~32]{NeQuRoUb} it is proved that quasi-copula $\ad$ is the exact upper bound for copulas with prescribed diagonal section $\delta$ in the case $\delta$ is a simple diagonal section. Recall that a diagonal section $\delta$ is called \emph{simple} if there exists $t_0 \in \II$ such that $\dhat$ is increasing on $[0,t_0]$ and decreasing on $[t_0,1]$, i.e., $\dhat$ is unimodal.
    Next step is to consider bimodal $\dhat$. It turns out that in this case the combination of values of $\dhat$ at its three local extrema plays an important role,
    where the signs in the combination alternate. This hints at the total variation of $\dhat$.
\end{remark}

\begin{proof}[Proof of Proposition~\ref{pr:bound}]
We choose any  $x,y \in \II$ with $x \le y$ and let $\varepsilon >0$.
By definition of $\tv_x^y$, there exists a sequence of points $x=x_0 < x_1 < \ldots < x_n = y$ such that
\begin{equation}\label{eq:TV_sup}
\sum_{i=1}^{n} | \dhat(x_{i}) - \dhat(x_{i-1})| \ge \tv_x^y(\dhat) - \varepsilon.
\end{equation}
Notice that as soon as $\dhat(x_{i-1}) \le \dhat(x_{i}) \le \dhat(x_{i+1})$ the point $x_i$ can be omitted from the partition while the sum in \eqref{eq:TV_sup} remains the same. The same is true if $\dhat(x_{i-1}) \ge \dhat(x_{i}) \ge \dhat(x_{i+1})$. So we want to find two subsequences $x=s_0 \le t_1 \le s_1 \le t_2 \le s_2 \le \ldots \le s_{m-1} \le t_m \le s_m =y$ of the sequence  $x=x_0 < x_1 < \ldots < x_n = y$ so that
\begin{equation}\label{eq:subsequence}
    \dhat(s_k) \le \dhat(t_k) \text{ and } \dhat(s_{k-1}) \le \dhat(t_k) \text{ for all } 1 \le k \le m,
\end{equation}
as demonstrated by example in Figure~\ref{fig:subseq}.
\begin{figure}
    \centering
    \includegraphics[width=\linewidth]{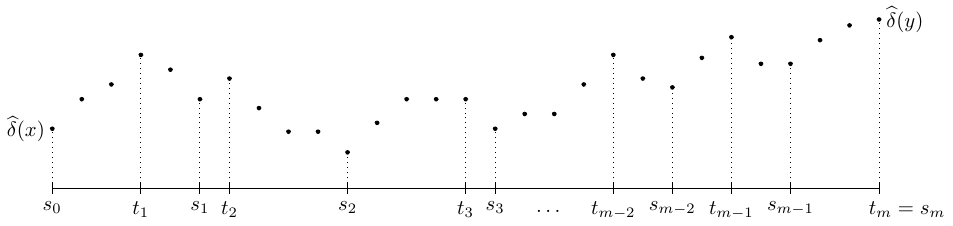}
    \caption{Two subsequences of the sequence $x=x_0 < x_1 < \ldots < x_n = y$ approximating the total variation of $\widehat{\delta}$ on $[x,y]$ in the proof of Proposition~\ref{pr:bound}.}
    \label{fig:subseq}
\end{figure}

Define
\begin{align*}
    && j_0 &=0, &s_0 & =x_0=x, \\
     A_1 &= \{ i \in \NN \colon 1 \le i \le n, \dhat(x_{i})-\dhat(x_{i-1}) < 0 \},   & i_1 &= \min A_1 -1, & t_1 &= x_{i_1}, \\
    &&& \text{ if } A_1 \neq \emptyset &&\\
    B_1 &= \{ i \in \NN \colon i_1+1 \le i \le n, \dhat(x_{i})-\dhat(x_{i-1}) >0 \},  & j_1 &= \min B_1-1,  & s_1 &= x_{j_1},\\
    &&& \text{ if } B_1 \neq \emptyset &&\\
    &\vdots & \vdots& &\vdots&\\
    A_k &= \{ i \in \NN \colon j_{k-1}+1 \le i\le n, \dhat(x_{i})-\dhat(x_{i-1}) < 0 \},  & i_k &= \min A_k-1,  & t_k &= x_{i_k}, \\
    &&& \text{ if } A_k \neq \emptyset &&\\
    B_k &= \{ i \in \NN \colon i_k+1 \le i \le n, \dhat(x_{i})-\dhat(x_{i-1}) > 0 \},  & j_k &= \min B_k-1, \quad & s_k &= x_{j_k}.\\
    &&& \text{ if } B_k \neq \emptyset &&\\
\end{align*}
Let $m$ be the first index such that either $A_{m}=\emptyset$
or $B_m=\emptyset$. If $A_{m} = \emptyset$ then we take $t_m=s_m=x_n=y$ and $i_m=j_m=n$. Else if $B_m=\emptyset$ we have $t_m$ already defined and we take $s_m=x_n=y$ and $j_m = n$.
In particular, if the last non-constant step of $\dhat$ is upwards, then $t_m=s_m$, and if the first non-constant step of $\dhat$ is downwards, then $s_0=t_1$. If $\dhat$ is constant on $\{x_i \colon i=0,1,\ldots,n\}$, then $m=1$, $s_0=x$ and $t_1=s_1=y$.

We have
\begin{align*}
    \dhat(t_k)-\dhat(s_{k-1}) &=\sum_{i=j_{k-1}+1}^{i_k} |\dhat(x_{i})-\dhat(x_{i-1})| \quad\text{and}\\
    \dhat(t_k)-\dhat(s_{k}) &=\sum_{i=i_{k}+1}^{j_k} |\dhat(x_{i})-\dhat(x_{i-1})|
\end{align*}
for all  $1 \le k \le m$, which implies
\begin{equation}\label{eq:4}
\displaystyle \sum_{k=1}^m \big(\dhat(t_k)-\dhat(s_{k-1})\big) +\sum_{k=1}^m \big(\dhat(t_k)-\dhat(s_{k})\big)= \sum_{i=1}^{n} | \dhat(x_{i}) - \dhat(x_{i-1})|.
\end{equation}
We now estimate the value of a copula $C$ at the point $(x,y)$ as follows. From the positivity of volumes of rectangles depicted in Figure~\ref{fig:red-green} (left) we obtain
\begin{align}
0 &\le \vol_C([x,t_1]\times[0,y]) + \sum_{k=1}^{m-1}\vol_C([t_k,t_{k+1}]\times[t_k,y]) + \vol_C([t_m,1]\times[t_m,y]) \label{eq:red}\\
& = C(t_1,y) - C(x,y) + \sum_{k=1}^{m-1}\big(C(t_{k+1},y) + \delta(t_k) - C(t_{k+1},t_k) - C(t_k,y)\big) \nonumber\\
& \qquad + y + \delta(t_m) - t_m - C(t_m,y) \nonumber\\
& = - C(x,y) + \sum_{k=1}^m\delta(t_k) - \sum_{k=1}^{m-1} C(t_{k+1},t_k) + y - t_m ,  \nonumber
\end{align}
so 
\begin{equation}\label{eq:5}
    C(x,y) \le \sum_{k=1}^m\delta(t_k) - \sum_{k=1}^{m-1} C(t_{k+1},t_k) + y - t_m .
\end{equation}

\begin{figure}
    \centering
    \includegraphics[height=0.35\textwidth]{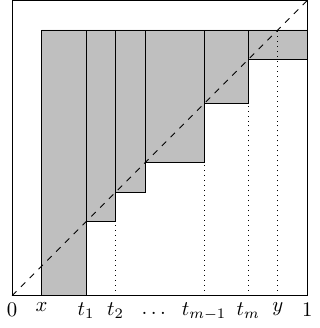}\qquad
    \includegraphics[height=0.35\textwidth]{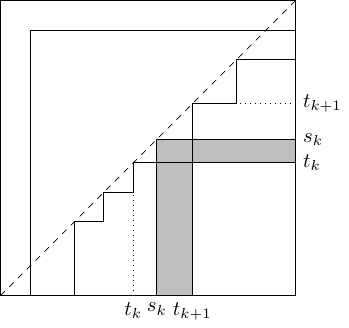}
    \caption{Disjoint unions of boxes used in the derivation of inequalities \eqref{eq:red} and \eqref{eq:green}.}
    \label{fig:red-green}
\end{figure}

Furthermore, for every $k=1,2, \ldots, m-1$ the positivity of volumes of rectangles depicted in Figure~\ref{fig:red-green} (right) implies
\begin{align}
0 &\le \vol_C([s_k,t_{k+1}]\times[0,s_k]) + \vol_C([t_{k+1},1]\times[t_k,s_k]) \label{eq:green}\\
& = C(t_{k+1},s_k) - \delta(s_k) +  s_k + C(t_{k+1},t_k) - t_k - C(t_{k+1},s_k) \nonumber\\
& =  \dhat(s_k) + C(t_{k+1},t_k) - t_k,  \nonumber
\end{align}
so $C(t_{k+1},t_k) \ge t_k - \dhat(s_k)$. It follows from inequality \eqref{eq:5} that
\begin{equation}\label{eq:indirect}
\begin{aligned}
C(x,y) &\le \sum_{k=1}^m\delta(t_k) - \sum_{k=1}^{m-1} C(t_{k+1},t_k) + y - t_m\\
&\le \sum_{k=1}^m\delta(t_k) - \sum_{k=1}^{m-1} \big(t_k - \dhat(s_k)\big) + y - t_m\\
& = \sum_{k=1}^m\delta(t_k) - \sum_{k=1}^m t_k + \sum_{k=1}^{m-1} \dhat(s_k) + y
 = y - \sum_{k=1}^m\dhat(t_k) + \sum_{k=1}^{m-1} \dhat(s_k)
\end{aligned}
\end{equation}
and by rearranging the terms we obtain
\begin{align*} 
    C(x,y) &\le y- \frac12 \left( \dhat(x) + \dhat(y) - \dhat(s_0) - \dhat(s_m) + 2\sum_{k=1}^m\dhat(t_k) - 2\sum_{k=1}^{m-1} \dhat(s_k) \right) \\
    &= y - \frac12 \left( \dhat(x) + \dhat(y) + \sum_{k=1}^m \big(\dhat(t_k)-\dhat(s_{k-1})\big) +\sum_{k=1}^m \big(\dhat(t_k)-\dhat(s_{k})\big) \right) \\
    &= y - \frac12 \left( \dhat(x) + \dhat(y) + \sum_{i=1}^{n} | \dhat(x_{i}) - \dhat(x_{i-1})| \right)\\
    & \le y - \frac12 \left(\dhat(x)+ \dhat(y) + \tv_x^y (\dhat) \right) + \frac{\varepsilon}{2},
\end{align*}
where the last equality and inequality follow from \eqref{eq:4} and \eqref{eq:TV_sup}.
We remark that the proof works also if $m=1$, in which case some of the sums in the above calculations are empty and should be understood as $0$.

By sending $\varepsilon$ to $0$, we obtain
$$C(x,y) \le y - \frac12 \left(\dhat(x)+ \dhat(y) + \tv_x^y (\dhat) \right)=\max \{ x,y \} - \frac12 \left(\dhat(x)+ \dhat(y) + \tv_{x \land y}^{x\lor y} (\dhat) \right).$$
In addition, $C(x,y) \le \min \{x,y\}$ since $C$ is a copula. It follows that the inequality \eqref{eq:zgornja} holds for any $x \le y$. If $x \ge y$ then $C(x,y)=C^t(y,x)$ and the same conclusion follows since the right hand side of \eqref{eq:zgornja} is symmetric in $x$ and $y$.
\end{proof}

For an interested reader we remark that the inequality
\begin{equation*}
C(x,y) \le \sum_{k=1}^m\delta(t_k) - \sum_{k=1}^{m-1} \delta(s_k) - \sum_{k=1}^m t_k + \sum_{k=1}^{m-1} s_k + y,
\end{equation*}
which is equivalent to inequality~\eqref{eq:indirect}, can also be obtained directly if we apply the theory of volumes of (formal) disjoint unions of rectangles and multiplicities of their vertices developed in \cite{OmSt20}. To this end one just needs to combine the rectangles depicted in Figure~\ref{fig:red-green} left with rectangles depicted in Figure~\ref{fig:red-green} right for all $k \in \{1,2,\ldots,m-1\}$, and calculate the $C$-volume of the so obtained formal disjoint union of rectangles. Since $C$ is a copula, this volume must be non-negative, which gives the above inequality  (see \cite{OmSt20} for details).

For a fixed diagonal section $\delta$ define an auxiliary function $f^\delta \colon\II^2 \to \RR$ with 
\begin{equation}\label{eq:f}
f^\delta(x,y)=y-\frac12 \big( \dhat(x)+\dhat(y)+ \tv_{x}^{y} (\dhat)\big),
\end{equation}
which we will often use in the proofs. As we will see, the function $f^\delta$ is non-zero, non-negative, has all $f^\delta$-volumes zero, but it is not grounded.

Next we show that the bound obtained in Proposition \ref{pr:bound} is attained by a copula for any point of the upper-left triangle of the unit square. 

\begin{theorem}\label{th:copula_U}
For any diagonal section $\delta$ the function $\ud \colon \II^2 \to \II$ defined by
$$\ud(x,y)=\min \left\{ x,y,y - \frac12 \left(\dhat(x)+ \dhat(y) + \tv_x^y (\dhat) \right) \right\}$$
is a copula with diagonal section $\delta$.
\end{theorem}

\begin{proof}
Note that $\ud(x,y)=\min \{ x,y,f^\delta(x,y) \}.$
Rewrite function $f^\delta$ as a sum of two univariate functions $f_1^\delta,f_2^\delta \colon \II \to \RR$ as follows
\begin{align*}
    f^\delta(x,y) &= y- \frac12 \dhat(x) -\frac12 \dhat(y) - \frac12 \tv_0^y (\dhat) +\frac12 \tv_0^x (\dhat) \\
    &= \left(-\frac12 \dhat(x) +\frac12 \tv_0^x (\dhat)\right) + \left(y- \frac12 \dhat(y)- \frac12 \tv_0^y (\dhat)\right) = f_1^\delta(x) + f_2^\delta(y),
\end{align*}
where 
$$ f_1^\delta(x) = -\frac12 \dhat(x) +\frac12 \tv_0^x (\dhat) \ \ \text{ and } \ \ f_2^\delta(y) = y- \frac12 \dhat(y)- \frac12 \tv_0^y(\dhat). $$

Now take an arbitrary rectangle $R=[a,b]\times [c,d] \subseteq \II^2$ and calculate $\vol_{f^\delta}(R)$. Using $f^\delta(x,y)=f_1^\delta(x)+f_2^\delta(y)$, it holds that $\vol_{f^\delta}(R)=0$. Since $R$ is arbitrary, the function $f^\delta$ is 2-increasing.

Next we show that the functions $f_1^\delta$ and $f_2^\delta$ are increasing and 1-Lipschitz. Choose $x_1,x_2 \in \II$ with $x_1\le x_2$.  Then
$$f_1^\delta(x_2)-f_1^\delta(x_1) = \frac12 \big( \tv_{x_1}^{x_2} (\dhat) - (\dhat(x_2)-\dhat(x_1)) \big) \ge 0,$$
hence $f_1^\delta$ is increasing. Note also that $\tv_0^x (\dhat)$ and $\dhat(x)$ are both 1-Lipschitz, so their difference is 2-Lipschitz, and if we apply multiplication with $\frac12$, we get that $f_1^\delta$ is 1-Lipschitz, as required. To show that $f_2^\delta$ is increasing, first choose $y_1,y_2 \in \II$ with $y_1 \le y_2$. Then, using the fact that $\tv_0^y (\dhat)$ and $\dhat(y)$ are 1-Lipschitz, we get
$$f_2^\delta(y_2)-f_2^\delta(y_1) = \frac12 \big( (y_2-y_1) -\tv_{y_1}^{y_2}( \dhat) \big) + \frac12 \big( (y_2-y_1) - (\dhat(y_2)-\dhat(y_1)) \big) \ge 0,$$
hence $f_2^\delta$ is increasing. To show that $f_2^\delta$ is 1-Lipschitz, we use the fact that $f_2^\delta$ can be written as a difference of two increasing 1-Lipschitz functions $y$ and $\frac12 (\dhat(y)+\tv_0^y(\dhat)$).
Since $f^\delta$ is $2$-increasing, increasing in each variable, and $1$-Lipschitz, we may apply \cite[Lemma~3.1]{DuJa} to conclude that $\ud(x,y)=\min \{ x,y,f^\delta(x,y) \}$ is 2-increasing.

Since $f_1^\delta(0)=f_2^\delta(0)=0$ and functions $f_1^\delta, f_2^\delta$ are increasing, they are non-negative. It follows that $\ud(x,y)=\min \{ x,y,f_1^\delta(x)+f_2^\delta(y) \}$ is grounded.
Similarly, using
\begin{align*}
    f^\delta(1,y) &= y- \frac12 \dhat(y)+ \frac12 \tv_y^1 (\dhat) = y+ \frac12 \big( \tv_y^1 (\dhat) - (\dhat(y) - \dhat(1)) \big) \ge y, \\
    f^\delta(x,1) &= 1- \frac12 \dhat(x) - \frac12 \tv_x^1 (\dhat) = x + \frac12 \big( \underbrace{ 1-x - \dhat(x)}_{\text{$\ge 0$}} \big) +\frac12 \big( \underbrace{ 1-x - \tv_x^1(\dhat)}_{\text{$\ge 0$}} \big)  \ge x,
\end{align*}
we see that $\ud$ has uniform marginals. Groundedness, 2-increasingness, and uniform marginals of $\ud$ imply that $\ud$ is a copula.

Finally, using
$$\ud(x,x)= \min \big\{ x,x, x- \frac12 \big( \dhat(x)+\dhat(x)+ \tv_x^x(\dhat) \big) \big\} = \min \{x, \delta(x) \} = \delta(x),$$
we see that the diagonal section of $\ud$ is equal to $\delta$.
\end{proof}

For any diagonal section $\delta$, using the function $f^\delta$ defined in \eqref{eq:f}, we introduce the sets
$$D_f^\circ(\delta) = \{(x,y) \in \II^2 \colon f^\delta(x,y) < \min\{x,y\}\}$$
and
\begin{align*}
    D_f(\delta) &=  \overline{D_f^\circ(\delta)} \cup \{(x,x) \in \II^2 \colon x\in \II\}, \\
    D_x(\delta) &=  \{(x,y) \in \II^2 \colon x \le y, x \le f^\delta(x,y)\}, \\
    D_y(\delta) &=  \{(x,y) \in \II^2 \colon y \le x, y \le f^\delta(x,y)\}, 
\end{align*}
where $\overline{D_f^\circ(\delta)}$ denotes the closure of the set $D_f^\circ(\delta)$.

Notice that the set $D_f(\delta)$ is not necessarily equal to the set $K=\{(x,y) \in \II^2 \colon f^\delta(x,y) \le \min\{x,y\}\}$.
Suppose for example that 
$$\delta(x)=
\begin{cases}
    0; & 0 \le x \le \frac14, \\
    2x-\frac12; & \frac14 < x \le \frac12, \\
    \frac12; & \frac12 < x \le \frac34, \\
    2x-1; & \frac34 < x \le 1.
\end{cases}$$
Assume that $(x,y) \in [0,\frac12] \times [\frac12,1]$. Since the function $f_1^\delta$ is $1$-Lipschitz and $f_2^\delta$ is increasing, we have 
$$\textstyle f^\delta(x,y) = f_1^\delta(x) + f_2^\delta(y) \ge f_1^\delta(\frac12) + x - \frac12 + f_2^\delta(\frac12) = x + f^\delta(\frac12,\frac12) - \frac12 = x.$$ 
It follows that $[0,\frac12) \times (\frac12,1] \cap D_f(\delta) = \emptyset$. On the other hand, $\tv_{\frac14}^{\frac34} (\dhat) = \frac12$, so that $f^\delta(\frac14,\frac34) = \frac14$ and $(\frac14,\frac34) \in K \setminus D_f(\delta)$.

If the whole rectangle $R=[a,b]\times[c,d]$ is included in either of the sets $D_f(\delta), D_x(\delta)$, or $D_y(\delta)$, then $\vol_{U_\delta}(R) =0$.
Since $D_f(\delta) \cup D_x(\delta) \cup D_y(\delta) = \II^2$, it follows that all the mass of copula $U_\delta$ is concentrated in the (common) boundary of the sets $D_f(\delta), D_x(\delta), D_y(\delta)$, i.e., in the set $\big(D_f(\delta) \cap D_x(\delta)\big) \cup \big(D_f(\delta) \cap D_y(\delta)\big)$.

Furthermore, we define functions $g^\delta_U, g^\delta_L \colon \II \to \II$ by
\begin{align*}
    g^\delta_U(x) &=  \max\{y \in \II \colon (x,y) \in D_f(\delta)\}, \\
    g^\delta_L(x) &=  \min\{y \in \II \colon (x,y) \in D_f(\delta)\}. 
\end{align*}
Note that max and min are attained since the set $D_f(\delta)$ is closed.
Since the main diagonal of the square $\II^2$ is included in $D_f(\delta)$, the functions $g^\delta_U$ and $g^\delta_L$ are well defined
and $g^\delta_L(x) \le x \le g^\delta_U(x)$ for all $x \in \II$.

\begin{lemma}\label{lem:g_U}
The sets $D_x(\delta)$, $D_y(\delta)$, $D_f(\delta)$ and functions $g^\delta_U$ and $g^\delta_L$ have the following properties: 
\begin{enumerate}[$(i)$]
\item For any point $(x,y) \in \II^2$ let $\Delta_{(x,y)}$ be a triangle with vertices $(x,y)$, $(x,x)$ and $(y,y)$. If $(x,y) \in D_f(\delta)$, then $\Delta_{(x,y)} \subseteq D_f(\delta)$.
\item The functions $g^\delta_U$ and $g^\delta_L$ are increasing. 
\item The set $D_f(\delta) \cap D_x(\delta)$ is a union of the graph of $g^\delta_U$ and the vertical segments $\{x\} \times [\lim_{t \nearrow x} g^\delta_U(t),\lim_{t \searrow x} g^\delta_U(t)]$ for all points $x$ where $g^\delta_U$ is not continuous.
Similarly, the set $D_f(\delta) \cap D_y(\delta)$ is a union of the graph of $g^\delta_L$ and the vertical segments $\{x\} \times [\lim_{t \nearrow x} g^\delta_L(t),\lim_{t \searrow x} g^\delta_L(t)]$ for all points $x$ where $g^\delta_L$ is not continuous.
\item The function $g^\delta_U$ is right-continuous and $g^\delta_L$ is left-continuous. 
\end{enumerate}
\end{lemma}

\begin{proof}
$(i)$:  The claim is clear if $x=y$. Now assume $x<y$ and let $(z,w) \in \II^2$ such that $x< z < w < y$, i.e., $(z,w)$ lies in the interior of $\Delta_{(x,y)}$. Since $(x,y) \in D_f(\delta)$ and $x<y$ it follows that $(x,y) \in \overline{D_f^\circ(\delta)}$, so there exists a sequence $(x_i,y_i) \in D_f^\circ(\delta)$ with $x_i\le y_i$ that converges to $(x,y)$. Hence, $f^\delta(x_i,y_i) < \min\{x_i,y_i\}=x_i$ and $x_i<z<w<y_i$ for sufficiently large $i$. The increasingness and $1$-Lipschitz property of $f^\delta$ imply 
$$f^\delta(z,w) \le f^\delta(z,y_i) \le f^\delta(x_i,y_i)+(z-x_i)<x_i+(z-x_i)=z,$$
so that $(z,w) \in D_f^\circ(\delta)$. Therefore, the interior of $\Delta_{(x,y)}$ is a subset of $D_f^\circ(\delta)$ and consequently $\Delta_{(x,y)} \subseteq D_f(\delta)$. The proof in case $x>y$ is similar.

$(ii)$: Suppose that $x < y$. If $g^\delta_U(x) \le y$, then $g^\delta_U(y) \ge y \ge g^\delta_U(x)$. So assume that $g^\delta_U(x) \ge y$. The point  $(x,g^\delta_U(x))$ belongs to $D_f(\delta)$ by definition, so that $\Delta_{(x,g^\delta_U(x))} \subseteq D_f(\delta)$ by $(i)$. Since $(y,g^\delta_U(x)) \in \Delta_{(x,g^\delta_U(x))}$ it follows that $g^\delta_U(y) \ge g^\delta_U(x)$. The proof for the function $g^\delta_L$ is similar.

$(iii)$: We first prove that the graph of $g^\delta_U$ is included in $D_f(\delta) \cap D_x(\delta)$. Suppose that the point $(x, y)$ belongs to the graph of $g^\delta_U$. Then $y = g^\delta_U(x)$ and $(x, y) \in D_f(\delta)$. Furthermore, for any $t > y$ we have $(x, t) \notin D_f(\delta)$, so $(x, t) \in D_x(\delta)$. Since $D_x(\delta)$ is closed, also $(x, y) \in D_x(\delta)$, so that $(x, y) \in D_f(\delta) \cap D_x(\delta)$. 

Let $x \in \II$ and define the set $A = \{t \in \II \colon (x,t) \in D_f(\delta) \cap D_x(\delta)\}$. It is closed and non-empty, since $D_f(\delta) \cap D_x(\delta)$ contains the graph of $g^\delta_U$. Let $y_1 = \min A$ and $y_2 = \max A$.  Since $(x, y_2) \in D_f(\delta)$, we have $f^\delta(x,y_2) \le x$ and since $(x, y_1) \in D_x(\delta)$, we have $f^\delta(x,y_1) \ge x$. Since  $f^\delta(x,y)$ is increasing in $y$, it follows that $f^\delta(x,y) = x$ for any $y \in  [y_1, y_2]$, so $\{x\} \times [y_1, y_2] \subseteq D_x(\delta)$. Furthermore, since $(x, y_2) \in D_f(\delta)$, we have $\{x\} \times [y_1, y_2] \subseteq D_f(\delta)$ by $(i)$. This implies that $A=[y_1, y_2]$.

Suppose that $\lim_{t \nearrow x} g^\delta_U(t) > y_1$. Then there exists $t < x$, such that $g^\delta_U(t) > y_1$. Since the point $(t, g^\delta_U(t)) \in D_f(\delta)$, it follows that $(x, y_1) \in D_f^\circ(\delta)$ by $(i)$, which is a contradiction with $(x, y_1) \in D_x(\delta)$.
It follows that $\lim_{t \nearrow x} g^\delta_U(t) \le y_1$. Furthermore, $\lim_{t \searrow x} g^\delta_U(t) \ge y_2$, since $(x,y_2) \in D_f(\delta)$ and thus $g^\delta_U(t) \ge y_2$ for any $t > x$ by $(i)$. It follows that $A \subseteq [\lim_{t \nearrow x} g^\delta_U(t),\lim_{t \searrow x} g^\delta_U(t)]$. On the other hand, both points $(x, \lim_{t \nearrow x} g^\delta_U(t))$ and $(x, \lim_{t \searrow x} g^\delta_U(t))$, belong to $D_f(\delta) \cap D_x(\delta)$, so also $[\lim_{t \nearrow x} g^\delta_U(t),\lim_{t \searrow x} g^\delta_U(t)] \subseteq A$. This finishes the proof for the set $D_f(\delta) \cap D_x(\delta)$. The proof for the set $D_f(\delta) \cap D_y(\delta)$ is analogous. 

$(iv)$: Suppose that the function $g^\delta_U$ is not continuous at $x \in \II$. Then the set $A = \{t \in \II \colon (x,t) \in D_f(\delta) \cap D_x(\delta)\}$
equals the interval $[\lim_{t \nearrow x} g^\delta_U(t),\lim_{t \searrow x} g^\delta_U(t)]$ by $(iii)$. If follows that $(x,\lim_{t \searrow x} g^\delta_U(t)) \in D_f(\delta)$, hence, $\lim_{t \searrow x} g^\delta_U(t) \le g^\delta_U(x)$. Furthermore, $\lim_{t \searrow x} g^\delta_U(t) \ge g^\delta_U(x)$ since $g^\delta_U$ is increasing. Thus $g^\delta_U$ is right-continuous at $x$.
The proof for the function $g^\delta_L$ goes the same way.
\end{proof}

Since a copula cannot have any mass on vertical segments, the above lemma implies that all the mass of copula $U_\delta$ is concentrated on the graphs of functions  $g^\delta_U$ and $g^\delta_L$.

Next example shows how the copula $\ud$ can be computed for a given diagonal section $\delta$. 

\begin{example}\label{ex:x^2}
Let $\delta$ be a diagonal section defined by $\delta(x)=x^2$ for all $x\in \II$. Then 
$$f_1^\delta(x)=\frac12 \big( \tv_0^x(\dhat) -\dhat(x) \big)=
\begin{cases}
    0; & x \le \frac12, \\
     x^2-x+\frac14; & x\ge \frac12.
\end{cases}$$
Similarly,
$$f_2^\delta(y)=y-\frac12 \big( \dhat(y)+ \tv_0^y (\dhat) \big) =
\begin{cases}
    y^2; & y \le \frac12, \\
    y- \frac14; & y \ge \frac12.
\end{cases}$$
We are now able to explicitly state $f^\delta(x,y)$:
$$f^\delta(x,y)=
\begin{cases}
    y^2; & (x,y) \in [0,\frac12] \times [0,\frac12], \\
    x^2-x+\frac14+y^2; & (x,y) \in [\frac12,1] \times [ 0, \frac12 ], \\
    y-\frac14; & (x,y) \in [0,\frac12] \times [\frac12,1], \\
    x^2-x+y; & (x,y) \in [\frac12,1] \times [\frac12,1],
\end{cases}$$
and $\ud(x,y)$:
$$\ud(x,y)=
\begin{cases}
    x; &  g^\delta_U(x) \le y, \\
    f^\delta(x,y); & g^\delta_L(x) \le y \le g^\delta_U(x), \\
    y; & y \le g^\delta_L(x),
\end{cases}
$$
where
$$ g^\delta_U(x) =
\begin{cases}
    \sqrt{x}; & 0 \le x \le \frac14, \\
    x+ \frac14; & \frac14 \le x \le \frac12, \\
    2x-x^2; & \frac12 \le x \le 1,
\end{cases}
\quad \textrm{ and } \quad
g^\delta_L(x)=
\begin{cases}
    0; & 0 \le x \le \frac12, \\
    \frac12 - \sqrt{x-x^2}; & \frac12 \le x \le 1.
\end{cases}
$$
\begin{figure}\label{fig:x_na_2}
\begin{center}
\includegraphics[align=c,width=0.35\textwidth]{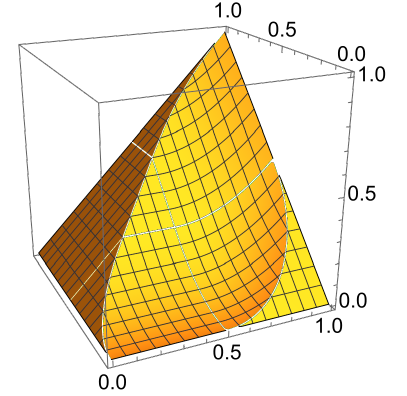}
\hspace{1cm}
\includegraphics[align=c,width=0.35\textwidth]{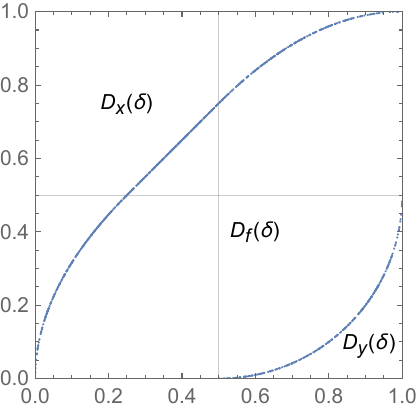}
\end{center}
\caption{The graph of copula $\ud$ with $\delta(x)=x^2$ (left) and its scatterplot with the corresponding regions $D_f(\delta)$, $D_x(\delta)$, and $D_y(\delta)$ (right). 
}
\end{figure}
\end{example}

We now collect our findings in our main theorem, which gives the formula for the exact upper bound for copulas with prescribed diagonal section, and thus settles an open question mentioned in the introduction.

\begin{theorem}\label{th:cc}
For any diagonal section $\delta$ we have
$$\cc(x,y) =  \min \left\{ x,y,\max \{ x,y \} - \frac12 \left(\dhat(x)+ \dhat(y) + \tv_{x \land y}^{x\lor y} (\dhat) \right) \right\}$$
for all $x,y \in \II$.
\end{theorem}

\begin{proof} Note that
\begin{equation*}
       \min \left\{ x,y, \max \{ x,y \} - \frac12 \left(\dhat(x)+ \dhat(y) + \tv_{x \land y}^{x\lor y} (\dhat) \right) \right\} =\max \{ \ud(x,y), U^t_\delta(x,y)\}   
\end{equation*}
for all $x,y \in \II$.
By Theorem~\ref{th:copula_U} the functions $\ud$ and $\ud^t$ are copulas. The result now easily follows from Proposition~\ref{pr:bound}.
\end{proof}

Note that in the language of Section~\ref{sec:asymmetry} quasi-copula $\cc$ is a diagonal splice of copulas $\ud$ and $\ud^t$, i.e.,
$$\cc(x,y)=\max \{ \ud(x,y), U^t_\delta(x,y)\}=(\ud \boxslash \ud^t) (x,y)=
\begin{cases}
        \ud(x,y); & x \le y, \\
        \ud^t(x,y); & y \le x.
     \end{cases}
     $$

The next example demonstrates the difference between the quasi-copulas $\kd$, $\cc$ and $\ad$.

\begin{example}\label{ex:KCA}
Let $\delta$ be a diagonal section defined by
$$
\delta(x)=
\begin{cases}
 0; & 0 \le x\le \frac{3}{10}, \\
 2x-\frac{3}{5}; & \frac{3}{10}\le x\le \frac{9}{20}, \\
 x-\frac{3}{20}; & \frac{9}{20}\le x\le \frac{11}{20}, \\
 \frac{2}{5}; & \frac{11}{20}\le x\le \frac{7}{10}, \\
 2 x-1; & \frac{7}{10}\le x \le 1. \\
\end{cases}
$$
The graph of the corresponding function $\dhat$ is shown in Figure~\ref{fig:CequalsA} (left).
As evident from the graphs depicted in Figure~\ref{fig:Kd_Cd_Ad}, quasi-copula $\cc$ is different from both $\kd$ and $\ad$ in this case, so in the point-wise order it lies strictly between them.
We remark that at point $(\frac{3}{10},\frac{7}{10})$, i.e., at one of the "indentations" of the graph of $\cc$, we have values $\kd(\frac{3}{10},\frac{7}{10})=\frac{1}{5}$, $\cc(\frac{3}{10},\frac{7}{10})=\frac{1}{4}$, and $\ad(\frac{3}{10},\frac{7}{10})=\frac{3}{10}$.
This point is one of the points where the difference between $\cc$ and $\ad$ is the biggest for this $\delta$.

\begin{figure}
    \centering
    \includegraphics[width=0.32\textwidth]{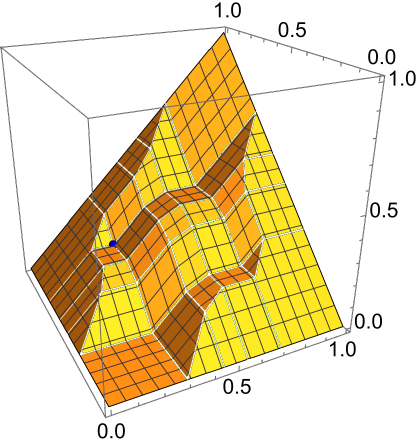}~
    \includegraphics[width=0.32\textwidth]{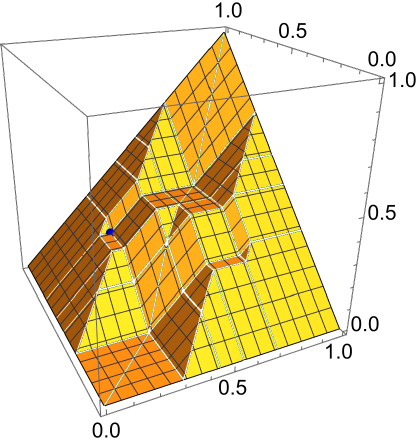}~
    \includegraphics[width=0.32\textwidth]{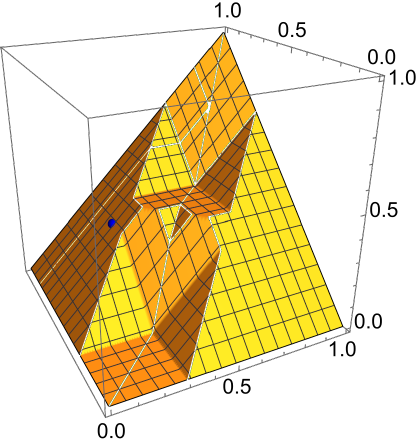}
    \caption{The graphs of functions $\kd$ (left), $\cc$ (middle), and $\ad$ (right) with $\delta$ from Example~\ref{ex:KCA}.}
    \label{fig:Kd_Cd_Ad}
\end{figure}
\end{example}

In \cite[Theorem 20]{NeQuRoUb} the authors prove that quasi-copula $\cc$ is a copula if and only if $\cc=\kd$, and pose a question to characterize diagonal sections $\delta$ for which this holds. Our next theorem answers this question.

\begin{theorem}\label{th:CequalK}
For any diagonal section $\delta$ the following are equivalent:
\begin{enumerate}[$(i)$]
\item $\cc=\kd$,
\item $\delta'(x) \in \{0,2\}$ almost everywhere on the set $\Lambda=\{ x \in \II \colon \delta(x) < x\}$. 
 \end{enumerate}
\end{theorem}

\begin{proof}
Suppose first that $(ii)$ holds.
The set $\Lambda$ is an open subset of $\II$, hence there exists a countable collection of disjoint open intervals $\{ \Lambda_i \}_{i\in J}$ such that $\Lambda= \bigcup_{i\in J} \Lambda_i$.

Let $C$ be any copula with $\delta_C = \delta$. Suppose that $(x,y) \in \II^2 \setminus \bigcup_{i\in J} (\Lambda_i\times \Lambda_i)$ and $x \le y$. 
Then either $x \in \II \setminus \bigcup_{i\in J} \Lambda_i$ or $x \in \Lambda_i$ for some $i\in J$ and $y \notin \Lambda_i$. 
In the first case $C(x,y) \ge C(x,x) = \delta(x) = x$, so $C(x,y) = M(x,y)$. 
In the second case let $\Lambda_i=(a,b)$. It follows that $y \ge b$. 
Since $C(a,a) = a$ and $C(b,b) = b$, we obtain that $C(a,b) \ge C(a,a) = a$ and $C(a,b) \le M(a,b) = a$, so that $C(a,b) = a$. Furthermore, since $x \in (a,b)$, we have $C(x,b) \le C(a,b) + x - a = x$ and $C(x,b) \ge C(b,b) + x - b = x$, so that $C(x,b) = x$.
It follows that $C(x,y) \ge C(x,b) = x$ and $C(x,y) = M(x,y)$ again. A similar argument shows that $C(x,y) = M(x,y)$ also for any $(x,y) \in \II^2 \setminus \bigcup_{i\in J} (\Lambda_i\times \Lambda_i)$ with $x \ge y$.

Therefore $\cc$ and $\kd$ match on the set $\II^2 \setminus \bigcup_{i\in J} (\Lambda_i\times \Lambda_i)$. Now let $i\in J$ and $\Lambda_i=(a,b)$.
Then $\delta(a)=a$ and $\delta(b)=b$, but $\delta(t)<t$ for all $t\in (a,b)$.
By assumption it holds that $\delta'(t) \in \{0, 2\}$ a.e. on $(a,b)$, or equivalently, $\dhat'(t) \in \{-1,1\}$ a.e. on $(a,b)$. By equation \eqref{eq:tvint} we have 
$$ \tv_x^y(\dhat) = \int_x^y |\dhat'(t)|dt = y-x $$
for every $x,y\in (a,b)$ with $x\le y$. It follows that $f^\delta(x,y) = \frac12 (\delta(x)+ \delta(y))$, so $\cc(x,y)=\kd(x,y)$ for every $(x,y) \in \Lambda_i\times \Lambda_i$.

Conversely, assume that $\kd=\cc$. Take any $x\in \Lambda$, so that $\delta(x)<x$. Recall the definition of function $f^\delta$ from equation \eqref{eq:f} and define a function $\varphi \colon (\max\{-x,x-1\}, \min\{x,1-x\}) \to \RR$ by $\varphi(t)=f^\delta(x-t,x+t)-(x-t)$. Since $x \in (0,1)$, function $\varphi$ is defined on a nonempty interval containing $0$. Function $\varphi$ is continuous and has value $\varphi(0)=f^\delta(x,x)-x=\delta(x)-x<0$. Hence, there exists $\varepsilon=\varepsilon_x>0$ such that $\varphi(\varepsilon)<0$. This implies $f^\delta(x-\varepsilon,x+\varepsilon)<x-\varepsilon<x+\varepsilon$, so that $(x-\varepsilon,x+\varepsilon) \in D_f(\delta)$.
Now,
$$\kd(x-\varepsilon,x+\varepsilon)=\cc(x-\varepsilon,x+\varepsilon)=f^\delta(x-\varepsilon,x+\varepsilon)<x-\varepsilon.$$
By Theorem~\ref{th:ABK} $(iii)$ we have $\kd(x-\varepsilon,x+\varepsilon)=\frac12 (\delta(x-\varepsilon)+ \delta(x+\varepsilon))$.
So $f^\delta(x-\varepsilon,x+\varepsilon)=\frac12 (\delta(x-\varepsilon)+ \delta(x+\varepsilon))$, and by the definition of $f^\delta$ we infer $\tv_{x-\varepsilon}^{x+\varepsilon}(\dhat) =2\varepsilon$. Since $|\dhat'(t)| \le 1$ where it exists (it exists a.e. on $\II)$, it follows from equation \eqref{eq:tvint} that $|\dhat'(t)|=1$ a.e. on the interval $(x-\varepsilon,x+\varepsilon)$. Hence, $\delta'(t) \in \{0,2\}$ a.e. on $(x-\varepsilon,x+\varepsilon)$.

The family $\mathcal{F}$ of intervals $(x-\varepsilon_x,x+\varepsilon_x)$ for all $x \in \Lambda$ is an open cover of $\Lambda$. Since $\Lambda$ is second-countable, there is a countable subfamily of $\mathcal{F}$ which still covers $\Lambda$.
It follows that $\delta'(x) \in \{0,2\}$ a.e. on $\Lambda$.
\end{proof}

In \cite[Theorem~32]{NeQuRoUb} it was shown that for a simple diagonal section $\delta$ (i.e., $\dhat$ unimodal), we have $\cc=\ad$. Next theorem gives a complete characterization of diagonal sections $\delta$ for which the quasi-copula $\cc$ is equal to $\ad$. 

\begin{theorem}\label{th:CequalA}
For any diagonal section $\delta$ the following are equivalent:
\begin{enumerate}[$(i)$]
\item $\cc=\ad$,
\item $y-x \ge \max\{\dhat(x),\dhat(y)\}$ whenever $x,y \in \II$ are such that $x<y$, $\dhat'(x),\dhat'(y)$ exist, and $\dhat'(x)<0<\dhat'(y)$.
 \end{enumerate}
\end{theorem}

\begin{proof}
$(i) \Rightarrow (ii)$: Suppose that $\cc=\ad$. Let $x,y \in \II$ be such that $x<y$, $\dhat'(x),\dhat'(y)$ exist, and $\dhat'(x)<0<\dhat'(y)$. We need to prove that $y-x \ge \max\{\dhat(x),\dhat(y)\}$. Let $t \in [x,y]$ be the point where the function $\dhat$ attains its maximum on $[x,y]$, i.e., $\dhat(t) = \max_{s\in [x,y]}\dhat(s)$. We will consider two cases, $t>x$ and $t=x$. 

Suppose first that $t>x$. Since $\dhat'(x)<0$, there exists $x_1>x$, such that $\dhat(x_1) < \dhat(x)$ and we can choose $x_1$ so that $x_1<t$. It follows that
\begin{align*}
    \tv_x^t(\dhat) &\ge |\dhat(x)-\dhat(x_1)| + |\dhat(x_1)-\dhat(t)| \\
    &= \dhat(x)-\dhat(x_1) + \dhat(t)-\dhat(x_1) \\
    &> \dhat(t)-\dhat(x),
\end{align*}
thus
\begin{align*}
    f^\delta(x,t) &= t - \frac12 \left(\dhat(x)+ \dhat(t) + \tv_x^t (\dhat) \right) \\
    &< t - \frac12 \left(\dhat(x)+ \dhat(t) +  \dhat(t)-\dhat(x) \right) \\
    &= t - \dhat(t).
\end{align*}
Now, 
$$ \cc(x,t) = \ad(x,t) = \min\{x, t - \max_{s\in [x,t]}\dhat(s)\} = \min\{x, t - \dhat(t)\}.$$
On the other hand, $\cc(x,t) = \min\{x, f^\delta(x,t)\}$. From $\min\{x, t - \dhat(t)\} = \min\{x, f^\delta(x,t)\}$ and $f^\delta(x,t) <  t - \dhat(t)$ it follows that $x=f^\delta(x,t)$, hence $\cc(x,t) = \ad(x,t) = x$. This implies $x < t - \dhat(t)$. Finally, we can estimate
$$ y-x \ge t-x \ge \dhat(t) \ge \max\{\dhat(x),\dhat(y)\}.$$

Suppose now that $t=x$. Since $\dhat'(y)>0$, there exists $y_1<y$, such that $\dhat(y_1) < \dhat(y)$ and we can choose $y_1$ so that $x<y_1$. Similarly as above we obtain 
$$ \tv_x^y(\dhat) > \dhat(x)-\dhat(y),$$
and
$$f^\delta(x,y) < y - \dhat(x).$$
Since
$$ \cc(x,y) = \ad(x,y) = \min\{x, y - \dhat(x)\}$$
it follows that $\cc(x,y) = \ad(x,y) = x$, thus $x \le y - \dhat(x)$ and finally
$$ y-x \ge \dhat(x) = \max\{\dhat(x),\dhat(y)\}.$$
   
$(ii) \Rightarrow (i)$: Let $(x,y) \in \II^2$. We need to prove that $\cc(x,y)=\ad(x,y)$. This clearly holds when $x=y$. Since $\cc$ and $\ad$ are both symmetric, we may thus assume that $x<y$. We consider two cases.

Suppose first that there are no points $x\le x_1 <t<y_1 \le y$ such that $\dhat(t)< \min\{ \dhat(x_1),\dhat(y_1)\}$. It follows that there exists $t_0 \in [x,y]$ such that $\dhat$ is increasing on $[x,t_0]$ (when $t_0>x$) and decreasing on $[t_0,y]$ (when $t_0<y$).
Hence, $\tv_x^y(\dhat)=2\dhat(t_0)-\dhat(x)-\dhat(y)$ and
\begin{align*}
    \cc(x,y) &=\min \big\{ x,y - \frac12 \big(\dhat(x)+ \dhat(y) + \tv_x^y (\dhat) \big) \big\}=
    \min \{ x,y - \dhat(t_0) \}\\
    &= \min \{ x,y - \max_{t \in [x,y]} \dhat(t)  \}=\ad(x,y).
\end{align*}

Now suppose there exist $x\le x_1 <t<y_1 \le y$ such that $\dhat(t)< \min\{ \dhat(x_1),\dhat(y_1)\}$.
Since $\dhat$ is continuous it has a global minimum on $[x_1,y_1]$. Let this global minimum be attained at $t_1$. By assumption $x_1<t_1<y_1$.
Let $[x_2,y_2] \subseteq [x_1,y_1]$ be the maximal closed interval such that $\dhat(t)=\dhat(t_1)$ for all $t \in [x_2,y_2]$.
It exists since $\delta$ is continuous. Clearly, $[x_2,y_2] \subseteq (x_1,y_1)$.

Fix some $\varepsilon >0$ such that $x_2-\varepsilon,y_2+\varepsilon \in [x_1,y_1]$.
We claim that there exists $0<\varepsilon_x<\varepsilon$ such that $\dhat'(x_2-\varepsilon_x)<0$. Suppose this is not the case. Then $\dhat'(t) \ge 0$ whenever $t \in (x_2-\varepsilon,x_2)$ and the derivative exists.
Equation~\eqref{eq:tvint} implies
$$0\le \tv_{x_2-\varepsilon}^{x_2}(\dhat)=\int_{x_2-\varepsilon}^{x_2} \dhat'(t)~dt=\dhat(x_2)-\dhat(x_2-\varepsilon)\le 0,$$
so that function $\dhat$ is constant on $[x_2-\varepsilon,x_2]$.
This contradicts the maximality of the interval $[x_2,y_2]$, which proves our claim.
A similar argument shows that there exists $0<\varepsilon_y<\varepsilon$ such that $\dhat'(y_2+\varepsilon_y)>0$.
The assumption in item $(ii)$ for the points $x_2-\varepsilon_x$ and $y_2+\varepsilon_y$ now implies
$$(y_2+\varepsilon_y)-(x_2-\varepsilon_x) \ge \max\{\dhat(y_2+\varepsilon_y),\dhat(x_2-\varepsilon_x)\} \ge \dhat(t_1),$$
hence, $y_2-x_2 \ge \dhat(t_1)-\varepsilon_x-\varepsilon_y >\dhat(t_1)-2\varepsilon$.
By sending $\varepsilon$ to $0$ we obtain $y_2-x_2 \ge \dhat(t_1)$ or equivalently $y_2 -\dhat(t_1) \ge x_2$.
Consequently, using Theorem~\ref{th:cc}, we obtain
\begin{align*}
\cc(x_2,y_2) &=\min \big\{ x_2,y_2 - \frac12 \big(\dhat(x_2)+ \dhat(y_2) + \tv_{x_2}^{y_2} (\dhat) \big) \big\}\\
&=\min \big\{ x_2,y_2 - \frac12 \big(\dhat(t_1)+ \dhat(t_1) + 0 \big) \big\}\\
&=\min \big\{ x_2,y_2 - \dhat(t_1) \big\}=x_2.
\end{align*}
Since $\cc$ is increasing in each variable and $1$-Lipschitz, we infer
$$\cc(x,y) \ge \cc(x,y_2) \ge \cc(x_2,y_2)-(x_2-x)=x.$$
This automatically implies that $\ad(x,y)=\cc(x,y)=x$.
\end{proof}

Condition $(ii)$ is clearly satisfied if the diagonal $\delta$ is simple (cf. \cite[Theorem~32]{NeQuRoUb}).
If the diagonal $\delta$ is not simple, there exist $x,y \in \II$ with $x<y$, such that $\dhat'(x),\dhat'(y)$ exist, and $\dhat'(x)<0<\dhat'(y)$. The condition $(ii)$ requires that any such $x$ and $y$ are far apart enough
relative to the values of $\dhat$.
More precisely, condition $\dhat'(x)<0<\dhat'(y)$ ensures that the global minimum of $\dhat$ on the interval $[x,y]$ is attained at some point $t_1 \in (x,y)$, and the proof of $(ii) \Rightarrow (i)$ shows that there exists a subinterval $[x_2,y_2] \subset (x,y)$, containing $t_1$, such that $\dhat$ is constant on $[x_2,y_2]$ and $y_2-x_2 \ge \dhat(x_2)$. So either $x_2<y_2$ and the graph of $\dhat$ has a flat section whose width is at least as large as its height, or $x_2=y_2$ and $\dhat$ has a zero on $(x,y)$.
In the latter case there exists $t \in (x,y)$ such that $\delta(t)=t$, so any copula $C$ with $\delta_C = \delta$ is an ordinal sum, see \cite[Theorem~3.2.1]{Ne99}.
Example of a diagonal $\delta$ that satisfies respectively violates condition $(ii)$ is depicted in Figure~\ref{fig:CequalsA}. \medskip

\begin{figure}
    \centering
    \includegraphics[align=c,width=0.4\textwidth]{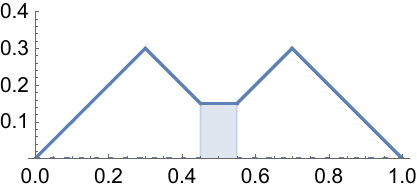}
    \hspace{1cm}
    \includegraphics[align=c,width=0.4\textwidth]{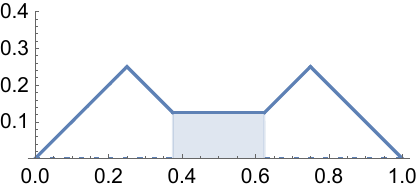}
    \caption{Upright versus lying shaded rectangle representing the condition $(ii)$ of Theorem \ref{th:CequalA} - in the first case we have $\cc \neq \ad$, while in the second we have $\cc=\ad$.}
    \label{fig:CequalsA}
\end{figure}

Next example shows that the dependence of the upper bound quasi-copula $\cc$ on the diagonal $\delta$ is not continuous. This is not surprising, since the formula contains total variation.

\begin{example} \label{ex:dhat_n}
Let $\delta$ be a diagonal section defined by
$$
\delta(x)=
\begin{cases}
 0; & 0 \le x\le \frac13, \\
 x-\frac13; & \frac13\le x\le \frac23, \\
 2x-1; & \frac23\le x \le 1. \\
\end{cases}
$$  
For every $n \in \NN$ we perturb $\delta$ into diagonal section $\delta_n = \delta - \varphi_n$, where $\varphi_n: \II \to \II$ is a zigzag function. More precisely, $\varphi_n$ is piecewise linear function with support equal to the interval $[\frac13,\frac23]$, for every $x \in [\frac13,\frac23]$ for which $\varphi'_n(x)$ exists we have $\varphi'_n(x) = \pm 1$, and $\varphi_n(x) \le \frac{1}{3n}$ for all $x \in \II$. It follows that the sequence of diagonal sections $\delta_n$ converges to $\delta$ in supremum norm. Since $\delta$ is a simple diagonal, we have $\cc = \ad$ by Theorem  \ref{th:CequalA}. Furthermore, $\delta'_n(x) \in \{0, 2\}$ whenever it exists, so $\overline{C}_{\delta_n} = K_{\delta_n}$ by Theorem  \ref{th:CequalK}. Now,
$$\lim_{n\to\infty}\overline{C}_{\delta_n} = \lim_{n\to\infty}K_{\delta_n} = \kd \ne \ad = \cc.$$
Figure \ref{fig:ex3.8} shows the graph of function $\dhat$, an example of perturbed $\dhat_n$, and graphs of quasi-copulas $\overline{C}_{\delta_n}$, $\cc=\ad$, and $\displaystyle\lim_{n\to\infty}\overline{C}_{\delta_n} = \kd$.
\end{example}

\begin{figure}
    \centering
    \includegraphics[align=c,width=0.4\textwidth]{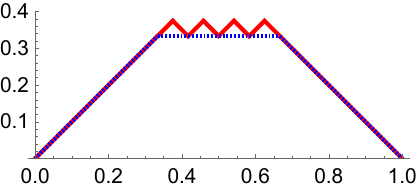}
    \includegraphics[align=c,width=0.35\textwidth]{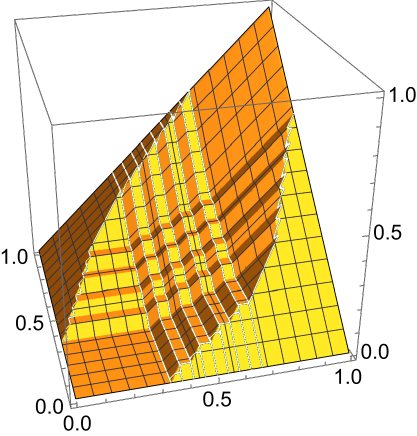} \\
    \includegraphics[align=c,width=0.35\textwidth]{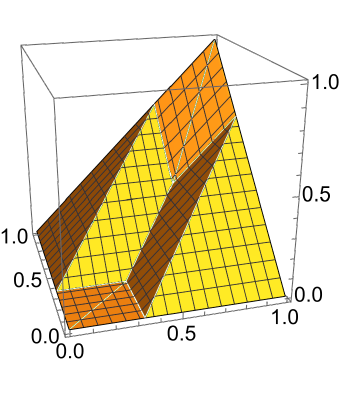}
    \includegraphics[align=c,width=0.35\textwidth]{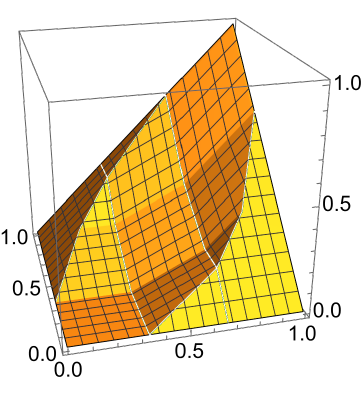}
      \caption{The graph of $\dhat$ in dotted line and an example of $\dhat_n$ in full line  (upper left), quasi-copula $\overline{C}_{\delta_n}$ (upper right), quasi-copula $\cc=\ad$ (lower left), and copula $\kd$ (lower right) from Example \ref{ex:dhat_n}.}
    \label{fig:ex3.8}
\end{figure}

\section{Maximal asymmetry of copulas with a given diagonal section}\label{sec:asymmetry}

In this section we apply our main results to the study of the asymmetry of copulas. We are interested in how large the asymmetry of a copula can be, if we know its diagonal section.

\begin{definition}
    For a given diagonal section $\delta$ we define
    $$\mu_\delta=\sup\{\mu(C)\colon C \in \mathcal{C}, \delta_C=\delta \},$$
    i.e., the maximal asymmetry of all copulas with diagonal section $\delta$. 
\end{definition}
We remark that the supremum is attained, see Corollary~\ref{co:asy}.
In order to understand the results of this section, we need from Section~\ref{sec:cop} the function $f^\delta$ given by equation \eqref{eq:f} and its properties, the copula $\ud$ from Theorem~\ref{th:copula_U}, and the function $g^\delta_U$ and its properties from Lemma \ref{lem:g_U}.

\medskip We recall the definition of the diagonal splice of two quasi-copulas with the same diagonal section, which is always a quasi-copula, see \cite{NeQuRoUb}.

\begin{definition} \label{def:4.2}
    Let $Q_1, Q_2 \in \mathcal{Q}$ with $Q_1(x,x)=Q_2(x,x)=\delta(x)$ for all $x\in \II$. For $x,y\in \II$ define
    $$(Q_1 \boxslash Q_2) (x,y)=
    \begin{cases}
        Q_1(x,y); & x\le y, \\ Q_2(x,y); & x\ge y.
    \end{cases}$$
\end{definition}

Next proposition shows that the diagonal splice of quasi-copula $\cc$ and Bertino copula $\bd$ is always a copula. 

\begin{proposition}
For a given diagonal section $\delta$ the function $\cc \boxslash B_\delta$ is a copula.
\end{proposition}

\begin{proof}
Since $\cc \boxslash B_\delta = \ud \boxslash B_\delta$, it suffices to show that $\ud(x,y)+ \bd(y,x) \le \delta(x)+ \delta(y)$ for every $x,y\in \II$ with $x \le y$ by \cite[Theorem 7]{NeQuRoUb}. Indeed, if $x \le y$ then
\begin{align*}
\delta(x)+\delta(y)- \ud(x,y)-\bd(y,x) &\ge \delta(x)+\delta(y)- \ud(x,y)-\ud(y,x) \\
&= \vol_{\ud}([x,y] \times [x,y] ) \ge 0,
\end{align*}
because $\ud$ is a copula.
\end{proof}

In the rest of the paper we will derive various formulas for calculating the maximal asymmetry of copulas with a given diagonal section. Each formula will be progressively easier to evaluate.
In particular, we show that
\begin{itemize}
\item maximal asymmetry $\mu_\delta$ is attained by copula $\cc \boxslash \bd$ (see Corollary~\ref{co:asy}),
\item the asymmetry of $\mu(\cc \boxslash \bd)$ is attained on the graph of the function $g_U^\delta$ (see Lemma~\ref{le:asy} and Theorem~\ref{th:asy}),
\item the asymmetry of $\mu(\cc \boxslash \bd)$ is attained on the intersection of the graph of $g_U^\delta$ and the extended graph of $h^\delta$ (see Proposition~\ref{pr:gdinH} and Theorem~\ref{th:asy1}),
\item for a simple diagonal we find a specific point $x_0 \in \II$ such that $\mu_\delta=\dhat(x_0)$ (see Proposition~\ref{pr:simple}).
\end{itemize}

\begin{corollary} \label{co:asy}
For a given diagonal section $\delta$ we have $\mu_\delta=\mu(\cc \boxslash B_\delta)$.
\end{corollary}

\begin{proof}
    For any copula $C$ with diagonal section $\delta$ and $0\le x \le y \le 1$ we have
    \begin{align*}
        |C(x,y)-C(y,x)| &= \max\{C(x,y),C(y,x)\}-\min\{C(x,y),C(y,x)\}\\
        &\le \cc(x,y)-\bd(y,x)\le \mu(\cc \boxslash B_\delta).
    \end{align*}
    Hence, $\mu(C) \le \mu(\cc \boxslash B_\delta)$. Since $\cc \boxslash B_\delta$ is a copula with diagonal section $\delta$, we conclude that $\mu_\delta= \mu(\cc \boxslash B_\delta)$.
\end{proof}

Now we look at the asymmetry of copula $\cc \boxslash B_\delta$.
We consider a vertical section $\{x_0\} \times [x_0,1]$, and show that the maximal asymmetry on this section is attained on its intersection with the graph of  $g_U^\delta$. Hence, $\mu(\cc \boxslash B_\delta)$ is attained on the graph of $g_U^\delta$.

\begin{lemma} \label{le:asy}
   For a given diagonal section $\delta$ and any $x_0\in \II$ we have 
   $$ \max_{y\in[x_0,1]}(\cc(x_0,y)-\bd(x_0,y)) = \cc(x_0,g^\delta_U(x_0))-\bd(x_0,g^\delta_U(x_0)) = \min_{t\in[x_0,g^\delta_U(x_0)]}\dhat(t).$$
\end{lemma}

\begin{proof}
    The second equality follows immediately from the definition of the function $g^\delta_U$ and Theorem~\ref{th:ABK} $(ii)$. 
    To prove the first equality we define a function $k\colon [x_0,1] \to \II$ by
    $$k(y) = \cc(x_0,y)-\bd(x_0,y).$$
    It is obvious that $k(x_0)=k(1)=0$. To prove the claim it is sufficient to show that $k$ is increasing on the interval $[x_0, g^\delta_U(x_0)]$ and decreasing on the interval $[g^\delta_U(x_0), 1]$. By the definition of the function $g^\delta_U$ we have for any $y \in [x_0, g^\delta_U(x_0)]$ that $(x_0, y) \in D_f(\delta)$, 
    and for any $y \in [g^\delta_U(x_0), 1]$ we have $(x_0, y) \in D_x(\delta)$. It follows that on the interval $[g^\delta_U(x_0), 1]$ we have
    $k(y) = x_0 - \bd(x_0,y)$, which is decreasing in $y$. On the interval $[x_0, g^\delta_U(x_0)]$ we have
    $$k(y) = f^\delta(x_0, y) - \bd(x_0,y) = y - \frac12\left(\dhat(x_0)+ \dhat(y) + \tv_{x_0}^y (\dhat) \right) - x_0 + \min_{t\in[x_0,y]}\dhat(t).$$
    Suppose that $x_0\le y_1 \le y_2\le g^\delta_U(x_0)$. Let $t_1 \in [x_0,y_1]$ be the point where $\min_{t\in[x_0,y_1]}\dhat(t) = \dhat(t_1)$ and $t_2 \in [x_0,y_2]$ be the point where $\min_{t\in[x_0,y_2]}\dhat(t) = \dhat(t_2)$. We have
    $$k(y_2) - k(y_1) = y_2 - y_1 - \frac12\left(\dhat(y_2)-\dhat(y_1) + \tv_{y_1}^{y_2} (\dhat) \right) + \dhat(t_2) - \dhat(t_1).$$
    Since $\dhat$ is 1-Lipschitz, we have $\tv_{y_1}^{y_2} (\dhat) \le y_2-y_1$, so 
    \begin{align*}
     k(y_2) - k(y_1) &\ge y_2 - y_1 - \frac12\left(\dhat(y_2)-\dhat(y_1) + y_2 - y_1 \right) + \dhat(t_2) - \dhat(t_1)   \\
     &= \frac12\left((y_2 - y_1) - \big(\dhat(y_2) - \dhat(y_1)\big) \right) + \dhat(t_2) - \dhat(t_1).
    \end{align*}
    If $\dhat(t_2) = \dhat(t_1)$, it follows that $k(y_2) - k(y_1) \ge 0$, since $\dhat$ is 1-Lipschitz. If $\dhat(t_2) < \dhat(t_1)$, then $y_1 < t_2$, and since $\dhat(t_1) \le \dhat(y_1)$ we obtain 
    \begin{align*}
     k(y_2) - k(y_1) &\ge \frac12\left(y_2 - y_1 - \dhat(y_2) + \dhat(y_1) \right) + \dhat(t_2) - \dhat(y_1)   \\
     &= \frac12\left(y_2 - t_2 - (\dhat(y_2) - \dhat(t_2)) + t_2 - y_1 - (\dhat(y_1) - \dhat(t_2)) \right) \ge 0
    \end{align*}
    since the function $\dhat$ is 1-Lipschitz. The conclusion follows.
\end{proof}

The following theorem is a direct consequence of Lemma~\ref{le:asy} and Corollary~\ref{co:asy}.

\begin{theorem}\label{th:asy}
       For a given diagonal section $\delta$  we have 
   $$ \mu_\delta = \max_{x\in\II} \min_{t\in[x,g^\delta_U(x)]}\dhat(t).$$
\end{theorem}

\begin{proof}
By Corollary~\ref{co:asy} we have $\mu_\delta=\mu(\cc \boxslash \bd)$.
Hence,
$$\mu_\delta= \max_{\substack{(x,y) \in \II^2\\ x \le y}} (\cc(x,y)-\bd(x,y))=
\max_{x \in \II} \max_{y \in [x,1]} (\cc(x,y)-\bd(x,y))=\max_{x\in\II} \min_{t\in[x,g^\delta_U(x)]}\dhat(t)$$
by Lemma~\ref{le:asy}.
\end{proof}

Next, we want to show that the asymmetry of $\cc \boxslash B_\delta$ is attained also on the (extended) graph of the function $h^\delta$, defined in \eqref{eq:h_delta}, corresponding to the support of Bertino copula $\bd$. Denote by $Z_\delta \subseteq \II$ the set of discontinuities of $h^\delta$. We define the set
$$ H_\delta = \{(x,h^\delta(x)) \in \II^2\colon x\in\II\} \cup \bigcup_{x\in Z_\delta } 
\big(\{x\}\times[\liminf_{t\to x}h^\delta(t), h^\delta(x)]\big),$$
i.e., the graph of the function $h^\delta$ with added vertical sections where it is discontinuous. 

\begin{proposition} \label{pr:gdinH}
Let $\delta$ be a diagonal section. If $(x,g^\delta_U(x))\in H_\delta$ for some point $x \in \II$, then $\displaystyle\min_{t\in[x,g^\delta_U(x)]}\dhat(t) = \dhat(x)$. There exist a point $x_0\in \II$, such that 
$$(x_0,g^\delta_U(x_0))\in H_\delta \quad\text { and } \quad \mu_\delta = (\cc \boxslash B_\delta)(x_0,g^\delta_U(x_0)) = \dhat(x_0).$$
\end{proposition}

\begin{proof}
Let $\tau\colon \II \to \II$ be a function defined by 
\begin{equation}\label{eq:tau}
\tau(x) = \min_{t\in[x,g^\delta_U(x)]}\dhat(t).
\end{equation}
Since $\bd$ and $\cc$ are continuous, also $\tau$ is continuous by Lemma~\ref{le:asy}. Suppose that $(x,g^\delta_U(x))\in H_\delta$ for some point $x \in \II$. Then $g^\delta_U(x) \le h^\delta(x)$. Thus for every $t\in [x, g^\delta_U(x)]$ it holds that $\dhat(t) \ge \dhat(x)$ by the definition of $\dhat(x)$. It follows that $\tau(x) = \dhat(x)$.

If $\delta=\delta_M$ then $g^\delta_U(x)=x$ and $h^\delta(x)=1$ for any $x \in \II$, so the claim is trivial. 
Suppose $\delta\ne\delta_M$, so that there exists $t\in \II$, such that $\dhat(t) = \varepsilon >0$. Since $\dhat$ is continuous, there exist points $x \in [0,t], y \in [t,1]$, such that $\dhat(x) = \dhat(y) = \frac{\varepsilon}{2}$ and $\dhat(s) \ge \frac{\varepsilon}{2}$ for any $s \in [x,y]$. It follows that $\bd(x,y) = x - \frac{\varepsilon}{2}$. Since $\varepsilon \le \tv_x^y (\dhat) \le y-x$, we have 
$$ f^\delta(x,y)=y-\frac12 \big( \dhat(x)+\dhat(y)+ \tv_{x}^{y} (\dhat)\big) \ge y-\frac12 \big( \varepsilon + y-x\big) = \frac12 \big(x+y -\varepsilon\big) \ge x,$$
so that $\cc(x,y) = x$. It follows that $\mu_\delta \ge \cc(x,y) - \bd(x,y) = \frac{\varepsilon}{2} >0$. 

Let $A$ be the set of all $x \in \II$ where $\tau$ attains its maximum $\mu_\delta$. The set $A$ is a closed subset of $\II$. Let $A'$ be its connected component. Then it is a closed interval $A'=[x_1,x_2]$, where we allow that $x_1=x_2$. Since $\mu_\delta>0$, we have $x_1,x_2 \notin \{0,1\}$. 

For any $\varepsilon >0$ there exist $t_1 \in \II \cap (x_1-\varepsilon, x_1)$ such that $\tau(t_1) <  \tau(x_1) = \mu_\delta$ and $t_2 \in \II \cap (x_2, x_2+\varepsilon)$ such that $\tau(t_2) < \tau(x_2)$. Since 
$$\tau(t_1) = \min_{s\in[t_1,g^\delta_U(t_1)]}\dhat(s) <  \tau(x_1) = \min_{s\in[x_1,g^\delta_U(x_1)]}\dhat(s)$$
and $g^\delta_U$ is increasing, it follows that there exists $t'_1 \in [t_1, x_1)$ such that $\dhat(t'_1) < \tau(x_1)$. By sending $\varepsilon$ to 0 we obtain $\tau(x_1) =\dhat(x_1)$. Therefore $\dhat(s) \ge \dhat(x_1)$ for any $s\in[x_1,g^\delta_U(x_1)]$, which implies  $h^\delta(x_1) \ge g^\delta_U(x_1)$. Furthermore, since
$$\tau(t_2) = \min_{s\in[t_2,g^\delta_U(t_2)]}\dhat(s) <  \tau(x_2) = \min_{s\in[x_2,g^\delta_U(x_2)]}\dhat(s),$$
it follows that there exists $t'_2 \in (g^\delta_U(x_2), g^\delta_U(t_2)]$ such that $\dhat(t'_2) < \tau(x_2)$. Since $\dhat(t'_2) < \tau(x_2) \le \dhat(x_2)$ it follows that $h^\delta(x_2) \le t'_2$. Since $g^\delta_U$ is right-continuous, by sending $\varepsilon$ to 0, we obtain $h^\delta(x_2) \le g^\delta_U(x_2)$.

If the component $A'$ is actually a singleton $x_0$, i.e., $x_1=x_2=x_0$, we obtain from the above $h^\delta(x_0)=g^\delta_U(x_0)$, so the point $(x_0,g^\delta_U(x_0))$ lies on the graph of $h^\delta$, thus in the set $H_\delta$.

Assume that $A'$ is not a singleton, so that $x_1<x_2$. If the function $g^\delta_U$ is continuous on $A'$, then its graph crosses $H_\delta$, since $g^\delta_U(x_1) \le h^\delta(x_1)$, $h^\delta(x_2) \le g^\delta_U(x_2)$, and $H_\delta$ is connected. So there exists $x_0 \in A'$, such that $(x_0,g^\delta_U(x_0)) \in H_\delta$.
Suppose now that the function $g^\delta_U$ has a jump at a point $x \in A'$. We will show that it is not possible that it jumps over $h^\delta(x)$. Indeed, suppose that 
$$y_1 = \lim_{t \nearrow x} g^\delta_U(t) \le h^\delta(x) < \lim_{t \searrow x} g^\delta_U(t) = g^\delta_U(x) = y_2.$$
Then the segment $\{x\}\times [y_1,y_2]$ is contained in the set $D_f(\delta) \cap D_x(\delta)$, so the continuity of the function $f^\delta$ implies $f^\delta(x,t)=x$ for all $t \in [y_1,y_2]$. In particular,
$$0=f^\delta(x,y_2)-f^\delta(x,y_1)=(y_2-y_1)-\frac12 \Big(\dhat(y_2)-\dhat(y_1)+\tv_{y_1}^{y_2}(\dhat) \Big).$$
Since $\tv_{y_1}^{y_2}(\dhat) \le y_2 - y_1$, this implies
$\dhat(y_2)-\dhat(y_1)=2(y_2-y_1)-\tv_{y_1}^{y_2}(\dhat) \ge y_2-y_1$. Since $\dhat$ is 1-Lipschits, it follows that $\dhat(y_2)-\dhat(y_1)=y_2-y_1$ and $\dhat$ is increasing on the interval $[y_1,y_2]$. Now, since $h^\delta(x) \ge y_1$, we have $\dhat(t) \ge \dhat(x)$ for all $t \in [x,y_1]$. Since $\dhat$ is increasing on the interval $[y_1,y_2]$, it follows that $\dhat(t) \ge \dhat(x)$ for all $t \in [x,y_2]$, so $h^\delta(x) \ge y_2$, a contradiction. 

We now have $g^\delta_U(x_1) \le h^\delta(x_1)$ and $h^\delta(x_2) \le g^\delta_U(x_2)$. We may assume that the latter inequality is strict, otherwise we are done by taking $x_0=x_2$.
Then the set $B=\{x \in [x_1,x_2] \colon h^\delta(x) < g^\delta_U(x)\}$ is nonempty, so we may define $x_0=\inf B$.
If $x_0=x_1$, then $g^\delta_U(x_0) \le h^\delta(x_0)$. If $x_0>x_1$ we can look at the left limit $\lim_{t \nearrow x_0} g^\delta_U(t) \le \limsup_{t \nearrow x_0} h^\delta(t) \le h^\delta(x_0)$, since $h^\delta$ is upper-semicontinuous. By the above it follows that $g^\delta_U(x_0) \le h^\delta(x_0)$ also in this case.
In particular, $x_0$ does not belong to the set $B$, so $h^\delta(t) < g^\delta_U(t)$ for some $t > x_0$ arbitrarily close to $x_0$.
It follows also that $x_0<x_2$, so we may calculate $\liminf_{t \to x_0} h^\delta(t) \le \liminf_{t \searrow x_0} h^\delta(t) \le \lim_{t \searrow x_0} g^\delta_U(t)=g^\delta_U(x_0)$.
This implies that $(x_0,g^\delta_U(x_0)) \in H_\delta$, which finishes the proof.
\end{proof}

Proposition \ref{pr:gdinH} (and its proof) provides an algorithm for computing $\mu_\delta$ for a given diagonal section~$\delta$ as follows:
\begin{enumerate}
    \item Compute the function $g^\delta_U$.
    \item Compute the function $h^\delta$.
    \item Add vertical sections to the graph of $h^\delta$ at jumps to determine $H_\delta$.
    \item Find the set of points $\Omega_\delta$ in the intersection of the graph of $g^\delta_U$ and $H_\delta$ that lie in the strict upper triangle $\{(x,y) \colon 0\le x<y \le 1\} \subseteq \II^2$.
    \item If $\Omega_\delta$ is empty, then $\mu_\delta=0$.
    \item Otherwise, find the maximal value of $\dhat(x)$ over all $x \in \II$ such that $(x,g^\delta_U(x)) \in \Omega_\delta$.
\end{enumerate}

\medskip
\noindent
A few remarks are in order:
    \begin{itemize}
        \item Note that the last step in the algorithm is valid due to the fact that $\tau(x) = \dhat(x)$ for all $(x,g^\delta_U(x)) \in \Omega_\delta$, where the function $\tau$ is defined in \eqref{eq:tau}.
        \item If a point $(x,y)$ lies on the diagonal of $\II^2$, then $|C(x,y)-C(y,x)|=0$ for any copula $C$. So if $\mu(C)>0$, then this point can be disregarded when computing $\mu(C)$. Thus, to narrow the set of candidates for computing $\mu_\delta$, we exclude diagonal points from $\Omega_\delta$ in step (4).        
        \item The case  $\delta_M(x)=x$ for all $x\in \II$ is a trivial one since  $B_{\delta_M}=M$ by definition of $\bd$. Therefore, the copula $\overline{C}_{\delta_M} \boxslash B_{\delta_M}$ equals $M$, and $\mu_{\delta_M}=0$.
        \item The set $\Omega_\delta$ is nonempty if and only if $\delta \neq \delta_M$. For any $\delta \neq \delta_M$ it holds that $\mu_\delta >0$.
    \end{itemize}

These findings are gathered in the main theorem of this section.

\begin{theorem} \label{th:asy1} 
       For a given diagonal section $\delta \ne \delta_M$  we have 
    $$\mu_\delta = \max \{\dhat(x)\colon x \in \II, (x,g^\delta_U(x)) \in H_\delta, g^\delta_U(x) > x\}.$$
\end{theorem}

The following example shows how we can compute the maximal possible asymmetry of all copulas with a given diagonal section using Theorem~\ref{th:asy1}. It also demonstrates that it is not enough to look just at the intersections of graphs of $g^\delta_U$ and $h^\delta$, we need to add vertical segments to the graph of $h^\delta$.

\begin{figure}
    \centering
    \includegraphics[align=c,width=0.4\textwidth]{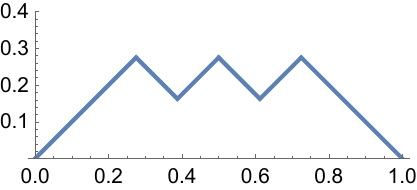}
    \hspace{1cm}
    \includegraphics[align=c,width=0.4\textwidth]{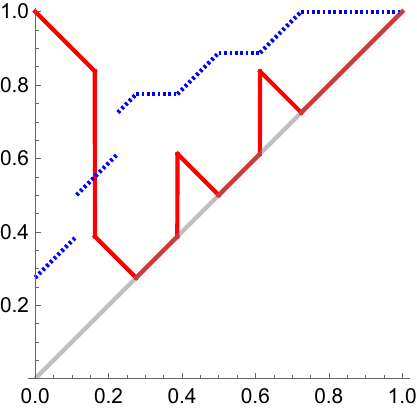}
     \caption{The graph of $\dhat$ (left), the set $H_\delta$ in full line and the graph of $g^\delta_U$ in dotted line (right) from Example \ref{ex:4.12}.}
    \label{fig:ex4.12}
\end{figure}

\begin{example} \label{ex:4.12}
Let $\delta$ be a diagonal section defined by 
$$\delta(x)=\begin{cases}
    0; & 0\le x\le \frac{11}{40},\\
    2x-\frac{11}{20}; & \frac{11}{40}\le x\le \frac{31}{80},\\
    \frac{9}{40}; & \frac{31}{80}\le x\le \frac12,\\
    2x-\frac{31}{40}; & \frac12\le x\le \frac{49}{80},\\
    \frac{9}{20}; & \frac{49}{80}\le x\le \frac{29}{40},\\
    2x-1; & \frac{29}{40} \le x\le 1.
\end{cases}$$
A lengthy calculation gives us
$$g^\delta_U(x)=\begin{cases}
    x+\frac{11}{40}; & 0\le x < \frac{9}{80},\\
    x+\frac{31}{80}; & \frac{9}{80}\le x < \frac{9}{40},\\
    x+\frac12; & \frac{9}{40}\le x < \frac{11}{40},\\
    \frac{31}{40}; & \frac{11}{40}\le x < \frac{31}{80},\\
    x+\frac{31}{80}; & \frac{31}{80}\le x < \frac12,\\
    \frac{71}{80}; & \frac12\le x < \frac{49}{80},\\
    x+\frac{11}{40}; & \frac{49}{80}\le x < \frac{29}{40},\\
    1; & \frac{29}{40} \le x\le 1, 
\end{cases} \text{\ \ \ and \ \ \ }  
h^\delta(x)=\begin{cases}
    1-x; & 0\le x \le \frac{13}{80},\\
    \frac{11}{20}-x; & \frac{13}{80} < x \le \frac{11}{40},\\
    x; & \frac{11}{40} < x < \frac{31}{80},\\
    1-x; & \frac{31}{80}\le x \le \frac12,\\
    x; & \frac12 < x < \frac{49}{80},\\
    \frac{29}{20}-x; & \frac{49}{80} \le x \le \frac{29}{40},\\
    x; & \frac{29}{40} < x\le 1.
\end{cases}$$
Apart from the graph of $h^\delta$ the set $H_\delta$ includes also vertical sections $\{\frac{13}{80}\} \times [\frac{31}{80},\frac{67}{80}], \{\frac{31}{80}\} \times [\frac{31}{80},\frac{49}{80}]$, and $\{\frac{49}{80}\} \times [\frac{49}{80},\frac{67}{80}]$. The only point in the strict upper triangle where $(x,g^\delta_U(x)) \in H_\delta$ is $x = \frac{13}{80}$, while $g^\delta_U(\frac{13}{80}) = \frac{11}{20} \ne h^\delta(\frac{13}{80}) = \frac{67}{80}$. It follows that
$\mu_\delta =\dhat(\frac{13}{80})= \frac{13}{80} = 0.1625$. Figure~\ref{fig:ex4.12} depicts the graph of the function $\dhat$ (left) and the graph of the function $g^\delta_U$ and the set $H_\delta$ (right).
\end{example}

In the special case that $\delta$ is a simple diagonal section Theorem~\ref{th:asy1} can be simplified as follows.

\begin{proposition} \label{pr:simple}
If $\delta$ is a simple diagonal section, then there is a unique point $x_0 \in \II$ such that $\dhat$ is increasing on $[0,x_0]$ and $(x_0,g^\delta_U(x_0)) \in H_\delta$. Furthermore, $\mu_\delta=\dhat(x_0)=\dhat(g^\delta_U(x_0))$.
\end{proposition}

\begin{proof}
Suppose $t_0 \in \II$ is the largest point at which $\dhat$ attains its global maximum, so that $h^\delta(t_0)=t_0\le g^\delta_U(t_0)$.
Since $h^\delta(0)=1 \ge g^\delta_U(0)$, we conclude as in the proof of Proposition~\ref{pr:gdinH} that there exists $x_0 \in [0,t_0]$ such that $(x_0,g^\delta_U(x_0)) \in H_\delta$.

Since the function $h^\delta$ is decreasing on $[0, t_0]$ and $g^\delta_U$ is increasing, there is only one such point $x_0$ with $(x_0,g^\delta_U(x_0)) \in H_\delta$, unless $h^\delta$ and $g^\delta_U$ are both constant on some interval $(x_1, x_2)$ with $x_1<x_2$. So suppose that this is the case. Since $h^\delta$ is constant, also $\dhat$ is constant on $(x_1, x_2)$. By definition of $g^\delta_U$ and continuity of $f^\delta$ we have 
$$f^\delta(t, g^\delta_U(t)) = f_1^\delta(t) + f_2^\delta(g^\delta_U(t)) = t$$ 
for every $t \in (x_1, x_2)$. Now, $f_1^\delta(t) = -\frac12 \dhat(t) +\frac12 \tv_0^t (\dhat)$ is constant on $(x_1, x_2)$, since $\dhat$ is, and also $f_2^\delta(g^\delta_U(t))$ is constant, since $g^\delta_U$ is constant here, which is a contradiction.

Since $h^\delta(t) \ge t_0$ for all $t \in \II$ and $\dhat$ is continuous and decreasing on $[t_0,1]$ we have
$$\dhat(\liminf_{t \to x_0} h^\delta(t))=\limsup_{t \to x_0} \dhat(h^\delta(t))=\limsup_{t \to x_0} \dhat(t)=\dhat(x_0)=\dhat(h^\delta(x_0)),$$
and consequently 
$\dhat$ is constant on $[\liminf_{t \to x_0} h^\delta(t),h^\delta(x_0)]$.
This interval contains $g^\delta_U(x_0)$ because $(x_0,g^\delta_U(x_0)) \in H_\delta$, hence,
\begin{equation}\label{eq:g(x_0)}
\dhat(g^\delta_U(x_0))=\dhat(h^\delta(x_0))=\dhat(x_0) \quad\text{and}\quad g^\delta_U(x_0) \ge t_0.
\end{equation}

Now let, $x_1 \in (t_0,1]$ be any point such that $(x_1,g^\delta_U(x_1)) \in H_\delta$, so that $t_0 \le g^\delta_U(x_1) \le h^\delta(x_1)$. Since $\dhat(h^\delta(x_1))=\dhat(x_1)$ and $\dhat$ is decreasing on $[t_0,1]$, it must be constant on $[x_1,h^\delta(x_1)]$.
But then $\dhat(g^\delta_U(x_1))=\dhat(x_1)$ because $x_1 \le g^\delta_U(x_1) \le h^\delta(x_1)$.
From this and from \eqref{eq:g(x_0)} we conclude that
$$\dhat(x_0)=\dhat(g^\delta_U(x_0)) \ge \dhat(g^\delta_U(x_1))=\dhat(x_1).$$
It now follows from Proposition~\ref{pr:gdinH} that $\mu_\delta=\dhat(x_0)$.
\end{proof}

Proposition~\ref{pr:simple} is illustrated in the following example.

\begin{example}
Let $\delta$ be a diagonal section defined by $\delta(x)=x^2$ for all $x\in \II$. Since $\delta$ is a simple diagonal and $\dhat$ is symmetric with respect to $x=\frac{1}{2}$, we have
$$h^\delta(x)=\begin{cases}
    1-x; & 0\le x\le \frac12,\\
    x; & \frac12<x\le 1,
\end{cases}$$
and $g^\delta_U$ is given in Example~\ref{ex:x^2}. Since $g^\delta_U(\frac38)=\frac58=h^\delta(\frac38)$ and $\dhat$ is increasing on $[0,\frac38]$, Proposition~\ref{pr:simple} implies
$\mu_\delta =\dhat(\frac38)= \frac{15}{64} \approx 0.2344$.
\end{example}

\section*{Acknowledgments}

The authors would like to thank the anonymous reviewers. Their suggestions helped us to greatly improve the presentation of the manuscript.

The authors acknowledge financial support from the ARIS (Slovenian Research and Innovation Agency, research core funding No. P1-0222).

\bibliographystyle{amsplain}
\bibliography{biblio}

\end{document}